\documentclass[a4paper, 11pt]{amsart}

\usepackage[T1]{fontenc}
\usepackage{lmodern}
\usepackage{microtype}
\usepackage[textwidth=15.7cm, textheight=22.7cm]{geometry}
\allowdisplaybreaks

\usepackage[colorlinks=true, urlcolor=blue, citecolor=red, linkcolor=blue, linktocpage,pdfpagelabels, bookmarksnumbered, bookmarksopen]{hyperref}

\usepackage[hyperpageref]{backref}
\usepackage[capitalize]{cleveref}
\usepackage{xcolor}
\usepackage{amsthm} 
\usepackage{latexsym}
\usepackage{amsmath}
\usepackage{amssymb}
\usepackage{esint}
\usepackage{accents}
\usepackage{enumitem}
\usepackage[colorinlistoftodos,prependcaption,textsize=tiny]{todonotes}

\usepackage{constants}

\usepackage{pgf,tikz}
\usepackage{mathrsfs}
\usetikzlibrary{arrows}
\usetikzlibrary{calc}
\usepackage{soul}
\usepackage{mathtools} 
\usepackage{xparse}

\usepackage[abbrev]{amsrefs}

\title[Traces of Sobolev maps]{Quantitative characterization of traces of Sobolev maps}

\author{Katarzyna Mazowiecka}
\address[Katarzyna Mazowiecka]{
Lehrstuhl f\"{u}r Angewandte Analysis, RWTH Aachen University, Pontdriesch 14-16, 52062 Aachen, Germany}
\email{mazowiecka@math1.rwth-aachen.de}

\author{Jean Van Schaftingen}
\address[Jean Van Schaftingen]{Universit\'e catholique de Louvain, Institut de Recherche en Math\'ematique et Physique, Chemin du Cyclotron 2 bte L7.01.02, 1348 Louvain-la-Neuve, Belgium}
\email{Jean.VanSchaftingen@uclouvain.be}

\subjclass[2010]{58D15 (46E35, 46T10, 46T20, 55S35)}

plus 6pt minus 12pt
\newcommand\eps{\varepsilon}

\newcommand\B{{\mathbb B}}

\newcommand\M{{\mathcal M}}
\newcommand\N{{\mathbb N}}

\DeclarePairedDelimiter{\brk}{(}{)}
\DeclarePairedDelimiter{\abs}{\lvert}{\rvert}
\DeclarePairedDelimiter{\set}{\lbrace}{\rbrace}
\DeclarePairedDelimiter{\norm}{\lVert}{\rVert}
\DeclarePairedDelimiter{\seminorm}{\lvert}{\rvert}
\DeclarePairedDelimiter{\floor}{\lfloor}{\rfloor}

\newcommand{\vast}{\bBigg@{3}}
\newcommand{\Vast}{\bBigg@{4}}
\newcommand{\deriv}{D}

\newcommand*\oline[1]{%
  \kern0.1em            
  \vbox{%
    \hrule height 0.5pt 
    \kern0.4ex          
    \hbox{%
      \kern-0.1em       
      $#1$
      \kern-0.1em       
    }
  }
  \kern0.1em            
}

\newtheorem{theorem}{Theorem}
\newtheorem{lemma}[theorem]{Lemma}

\newtheorem{proposition}[theorem]{Proposition}

\theoremstyle{definition}

\theoremstyle{remark}
\newtheorem{remark}[theorem]{Remark}

\newcommand\dist{\mathrm{dist\,}}

\newcommand\lip{\mathrm{Lip\,}}

\newcommand{\R}{\mathbb{R}}

\newcommand{\Z}{\mathbb{Z}}

\newcommand{\brac}[1]{\left (#1 \right )}
\newcommand{\tnorm}[1]{{\left\vert\kern-0.25ex\left\vert\kern-0.25ex\left\vert #1 \right\vert\kern-0.25ex\right\vert\kern-0.25ex\right\vert}} 

\newcommand{\n}{\mathcal{N}}
\newcommand{\m}{\mathcal{M}}
\newcommand{\compose}{\,\circ\,}
\newcommand{\defeq}{\coloneqq} 
\newcommand{\dif}{\,\mathrm{d}}

\newcommand{\Barcal}{\mbox{$ \rule[.036in]{.11in}{.007in}\kern-.128in\int $}}

\def\mvint_#1{\mathchoice
          {\mathop{\vrule width 6pt height 3 pt depth -2.5pt
                  \kern -8pt \intop}\nolimits_{\kern -3pt #1}}%

          {\mathop{\vrule width 5pt height 3 pt depth -2.6pt
                  \kern -6pt \intop}\nolimits_{#1}}%
          {\mathop{\vrule width 5pt height 3 pt depth -2.6pt
                  \kern -6pt \intop}\nolimits_{#1}}%
          {\mathop{\vrule width 5pt height 3 pt depth -2.6pt
                  \kern -6pt \intop}\nolimits_{#1}}}

\numberwithin{theorem}{section} \numberwithin{equation}{section}

\newcommand{\tr}{\operatorname{tr}}

\definecolor{ttzzqq}{rgb}{0.2,0.6,0}
\definecolor{ffqqqq}{rgb}{1,0,0}
\definecolor{gr}{rgb}{0.2,0.6,0}
\definecolor{re}{rgb}{1,0,0}
\definecolor{bl}{rgb}{0,0,0.8}

\tikzset{>=stealth}

\begin{document}

\begin{abstract}
We give a quantitative characterization of traces on the boundary of Sobolev maps in $\dot{W}^{1,p}(\mathcal M, \mathcal N)$, where $\mathcal{M}$ and $\mathcal{N}$ are compact Riemannian manifolds, $\partial \m \neq \emptyset$: the Borel--measurable maps $u\colon \partial \mathcal M \to \mathcal{N}$ that are the trace of a map $U\in \dot{W}^{1,p}(\mathcal M, \mathcal{N})$ are characterized as the maps for which 
there exists an extension energy density $w \colon \partial \mathcal{M} \to [0,\infty]$ that controls the Sobolev energy of extensions from $\lfloor p - 1 \rfloor$--dimensional subsets of $\partial \mathcal{M}$ to $\lfloor p\rfloor$--dimensional subsets of $\mathcal{M}$.
\end{abstract}
\maketitle

\section{Introduction}
Given \(\m\) a compact Riemannian manifold with non-empty boundary \(\partial \m\), we consider the homogeneous Sobolev space defined as
\[
 \dot{W}^{1, p} (\m, \R)
 \defeq 
 \biggl\{U \colon \m \to \R \colon U \text{ is weakly differentiable and } \deriv U \in L^p(\m)\biggr\}.
\]
The classical trace theorem of E.~Gagliardo \cites{Gagliardo} states that for \(p > 1\) there is a well-defined continuous and surjective trace operator 
\[
  \tr_{\partial \m} 
  \colon 
  \dot{W}^{1, p} (\m, \R) 
  \to 
  \dot{W}^{1-1/p, p} (\partial \m, \R), 
\]
such that for functions $U$ that are additionally continuous we have \(\tr_{\partial \m} U = U \big\vert_{\partial \m}\).
Here, for $0<s<1$ and $p\ge 1$, $\dot{W}^{s,p}(\partial \m, \R)$ is the homogeneous Sobolev--Slobodeckij space, or fractional Sobolev space, defined as 
\begin{equation}
\label{eq_gae7ieCheexiwoushaibail6}
  \dot{W}^{s, p} (\partial \m, \R)
 \defeq
 \set[\bigg]{u \colon \m \to \R \colon \int_{\partial \m}\int_{\partial \m} \frac{\abs{u (x) - u (y)}^p}{d_{\partial \m}(x,y)^{m - 1 + sp}} \dif x \dif y < \infty}, 
\end{equation}
where \(m \defeq \dim \m = \dim \partial \m + 1\) and $d_{\partial \m}$ is the geodesic distance on $\partial \m$.

If \(\n\) is a compact Riemannian manifold, that by J.~Nash's embedding theorem \cite{Nash} can be assumed without loss of generality to be isometrically embedded into some Euclidean space $\R^\nu \supseteq \n$, then the homogeneous spaces of Sobolev mappings can be defined for \(p \ge 1\) as 
\[
 \dot{W}^{1,p}(\m,\n) \defeq \set[\big]{U\in \dot{W}^{1,p}(\m,\R^\nu)\colon U(x) \in \n \text{ for almost every } x\in \m}
\]
and for \(0 < s < 1\) and \(p \ge 1\)
\[
  \dot{W}^{s, p} (\partial \m,\n) 
  \defeq 
  \set{u\in \dot{W}^{s, p} (\partial \m,\R^\nu)\colon u(x) \in \n \text{ for almost every } x\in \partial \m};
\]
these nonlinear Sobolev spaces arise naturally, for example, as domains of functionals in the calculus of variations and of partial differential equations in geometric analysis and physical models.

As a consequence of the straightforward vector version of Gagliardo's trace theorem, the trace operator \(\tr_{\partial \m}\) is well-defined and continuous from \(\dot{W}^{1, p} (\m, \n)\) to \(\dot{W}^{1-1/p, p} (\partial \m, \n)\). 
The question of surjectivity of the trace operator is however much more delicate: given a map \(u \in \dot{W}^{1-1/p, p} (\partial \m, \n)\), the classical linear extension  construction gives a function \(U \in \dot{W}^{1, p} (\partial \m, \R^\nu)\) such that \(\tr_{\partial \m} U = u\) with no guarantee whatsoever about the range of the extension \(U\).

Indeed, the surjectivity of the trace operator can first fail because of \emph{global topological obstructions}: For instance if \(p > m\), by the Morrey--Sobolev embedding, mappings in the spaces \(\dot{W}^{1, p} (\m, \n)\) and \(\dot{W}^{1 - 1/p, p}(\partial \m, \n)\) are almost everywhere equal to continuous maps, and classical topological obstructions for the extension of continuous maps results in obstructions for the extension of Sobolev mappings.
When \(1 \le p < m\), a \emph{Lipschitz--continuous} map $u\in \lip(\partial \m, \n)$ is known to be a trace of a map in \(\dot{W}^{1, p} (\m, \n)\) if and only if $u$ has a continuous extension to $\partial \m \cap \m^{\floor{p}}$, where $\m^{\floor{p}}$ is a $\lfloor p\rfloor$--dimensional skeleton of $\m$ \cite{White-homotopy}*{Section 4}; here and in the sequel \(\floor{t}\) denotes the integer part of the real number \(t\), so that \(\floor{t} \in \Z\) and \(\floor{t} \le t < \floor{t} + 1\).

\emph{Local topological obstructions} can prevent locally the surjectivity of the trace operator: If \(p < m\) and if the homotopy group \(\pi_{\floor{p - 1}} (\n)\) is not trivial, by definition there exists a map \(f \in C^\infty (\mathbb{S}^{\floor{p - 1}}, \n)\) which is not homotopic to a constant; define the mapping \(u \colon \mathbb{B}^{m - 1} \to \n\) for \(x = (x', x'') \in \mathbb{B}^{m - 1} \subset \R^{\floor{p}} \times \R^{m - \floor{p + 1}}\) by \(u (x) \defeq f (x'/\abs{x'})\); we have \(u \in \dot{W}^{1-1/p, p} (\partial \m, \n)\), whereas there is no \(U \in \dot{W}^{1, p} (\mathbb{B}^{m - 1} \times (0, 1), \n)\) such that \(\tr_{\mathbb{B}^{m - 1}} U = u\) \citelist{\cite{HL1987}*{Section 6.3}\cite{Bethuel-Demengel}*{Theorem 4}}.

\emph{Analytical obstructions} finally arise locally for the extension problem:
There exist maps in \(\dot{W}^{1 - 1/p, p} (\B^{m - 1}, \n)\) that are strong limits of smooth maps from \(\B^{m - 1}\) to \(\n\) but are not traces of maps in \(\dot{W}^{1, p} (\B^{m - 1}\times (0, 1), \n)\). This is known to happen when either the homotopy group \(\pi_{\ell} (\n)\) is infinite for some \(\ell \in \N\) with \(\ell \le \max\{m, p\} - 1\) \cite{Bethuel-obstruction} (see also \cite{Bethuel-Demengel}*{Theorem 6}) and when \(p \in \N \setminus \{0, 1\}\) and the homotopy group \(\pi_{p-1} (\n)\) is nontrivial \cite{Mironescu_VanSchaftingen_trace}. 
These analytical obstructions can be seen in view of a nonlinear uniform boundedness principle as a consequence of the failure of linear estimates on extensions for smooth maps \cite{Monteil_VanSchaftingen}; when \(2 \le p < 3\), these analytical obstructions are connected to similar analytical obstructions for the lifting problem in fractional Sobolev spaces \citelist{\cite{Mironescu_VanSchaftingen_Lifting}\cite{Bethuel_Chiron}}.

On the other hand, the trace is known to be surjective from \(\dot{W}^{1, p} (\B^{m - 1} \times (0, 1),\n)\) onto \(\dot{W}^{1-1/p, p}(\B^{m-1},\n)\) in the following cases: when \(p \ge m\) \cite{Bethuel-Demengel}*{Theorem 1 \& 2}, for \(p > m\) it is a consequence of the Morrey--Sobolev embedding whereas for \(p = m\) it is a consequence of the embedding into maps of vanishing mean oscillation (VMO) \cites{Brezis_Nirenberg_1, Brezis_Nirenberg_2}; when \(1 < p < 2\le m\) or \(2 \le p < 3\le m\) with $\pi_1(\n)\simeq \{0\}$ \cite{HL1987}*{Theorem 6.2}; when \(3 \le p < m\), \(\pi_1 (\n)\) is finite, and \(\pi_2(\n) \simeq \dotsb \simeq \pi_{\floor{p -1}}(\n) \simeq \{0\}\) \citelist{\cite{HL1987}*{Theorem 6.2}\cite{Mironescu_VanSchaftingen_trace}}. The case when \(4 \le p < m\), \(\pi_1 (\n), \dotsc, \pi_{\floor{p - 2}} (\n)\) are finite, and \(\pi_{\floor{p - 1}}(\n)\simeq \{0\}\) remains open.

We are interested in the question of characterizing in general the range of the trace operator.
T.~Isobe \cite{Isobe} has provided characterization of the maps $u\colon \partial \m \to \n$ that are the traces of maps in \(\dot{W}^{1, p}(\mathcal M,\n)\) as the maps satisfying the two conditions: 
\begin{enumerate}[label=(\(\mathfrak{o}_\Alph*\))]
 \item 
 \label{it_zuugh7Thai5ThiephahtoP8y}
 The mapping \(u\) satisfies
 \[
  \varlimsup_{\substack{t \to 0}}\varlimsup_{\varepsilon \to 0}
  \inf\set*{\smashoperator[r]{\int_{[0,t)\times \partial \mathcal M}}\abs{\deriv U}^p + \frac{\operatorname{dist} (U,\mathcal N)^p}{\eps^p}  \colon U\in \dot{W}^{1,p}([0,t)\times\partial \mathcal M, \R^\nu) \text{ and } \tr_{\partial \mathcal M} U =u} = 0.
 \]
 \item 
 \label{it_dahGaiquu0xejai3Tah8aipa} The restriction of the mapping \(u\) to a generic triangulation $\mathcal M^{\lfloor p-1 \rfloor}\cap \partial \mathcal M$ is homotopic in \(\mathrm{VMO}(\mathcal M^{\lfloor p-1 \rfloor}\cap \partial \mathcal M,\n)\) to the restriction of a continuous function from $\mathcal M^{\floor{p}}$ to \(\n\), where $\mathcal M^\ell$ with \(\ell \in \{0, \dotsc, m\}\) stands for the $\ell$--dimensional skeleton of a triangulation of $\mathcal M$.
\end{enumerate}
Isobe's first obstruction \ref{it_zuugh7Thai5ThiephahtoP8y} is equivalent to \(u\) belonging to the image of the trace operator on \(\dot{W}^{1, p} (\partial \m \times [0, 1], \n)\),
implying in particular through classical linear trace theory that \(u \in W^{1-1/p, p} (\partial \m, \n)\); 
this condition \ref{it_zuugh7Thai5ThiephahtoP8y}  can be seen as an asymptotic condition on a family of Ginzburg--Landau functionals; according to Isobe the problem of characterizing in general maps satisfying \ref{it_zuugh7Thai5ThiephahtoP8y} remains open \cite{Isobe}*{p.\ 367}. 
In \ref{it_dahGaiquu0xejai3Tah8aipa}, the notion of generic skeleton has to be understood in the sense of holding for almost every value of the parameter for a parametrized family of triangulations (see \citelist{\cite{Hang_Lin_II}*{Section 3}\cite{White-homotopy}*{Section 3}\cite{White-infima}*{Section 3}}); the homotopy in \(\mathrm{VMO}(\mathcal M^{\lfloor p-1 \rfloor}\cap \partial \mathcal M,\n)\) is understood in the sense of \cite{Brezis_Nirenberg_1} (see also \cref{section_qualitative}). When \(p \not \in \N\), \ref{it_dahGaiquu0xejai3Tah8aipa} can be simplified into requiring that restrictions of the mapping \(u\) to generic triangulations $\mathcal M^{\lfloor p-1 \rfloor}\cap \partial \mathcal M$ are equal almost everywhere to restrictions of continuous mappings from $\mathcal M^{\floor{p}}$ to \(\n\).

The goal of the present work is to characterize the image of the trace by the properties of mappings on lower-dimensional subsets. This approach is motivated by the fact that in the Gagliardo energy appearing in the definition \eqref{eq_gae7ieCheexiwoushaibail6}
of the fractional Sobolev space \(\dot{W}^{1-1/p, p} (\partial \m, \n)\), the quotient 
\[
\frac{\abs{u (x) - u (y)}^p}{d_{\partial \m}(x, y)^{p - 1}}
\]
can be interpreted as the minimal energy in \(\dot{W}^{1, p} ([0, d_{\partial \m}(x, y)], \R)\) to connect \(u (x)\) to \(u (y)\). 
Because of the quantitative nature of the phenomenon of analytical obstructions, we expect any characterization of the trace space to have some quantitative character. Finally, a workable characterization should be based on a robust definition of generic lower-dimensional set, as developed as topological screening by P.~Bousquet, A.~Ponce, and J.~Van Schaftingen \cite{BPVS}.

We first consider the case where the domain manifold \(\m\) is the \(m\)--dimensional half-space \(\R^m_+ \defeq \R^{m - 1} \times (0, \infty)\) with boundary \(\partial \R^m_+ = \R^{m - 1} \times \{0\}\) and closure \(\oline{\R}^m_+ \defeq \R^{m - 1} \times [0, \infty)\). In order to formulate our results we settle some terminology and notation.

We assume that simplexes of a simplicial complex inherit the metric and the measure from their canonical realization as an equilateral simplex of side-length~\(1\); on the full complex \(\Sigma\), a measure is defined by additivity and a distance \(d^\Sigma\).  Given simplicial complexes $\Sigma$ and $\Sigma_0$ and \(\lambda > 0\), we define the quantity
\begin{equation}\label{eq:Alhforslike}
 \gamma^{\lambda}_{\Sigma_0, \Sigma}
 \defeq \sup_{\substack{z \in \Sigma_0 \\ \delta > 0}} \frac{\abs{B^\Sigma_{\lambda \delta} (z)} }{\delta \abs{\Sigma_0 \cap B^\Sigma_{\delta} (z)}},
\end{equation}
with the measure in the numerator being taken relatively to \(\Sigma\) and in the denominator relatively to \(\Sigma_0\); we note that \(\gamma^{\lambda}_{\Sigma_0, \Sigma} <\infty\) for every \(\lambda > 0\) whenever \(\Sigma\) is a finite homogeneous simplicial complex and if \(\Sigma_0 \subset \Sigma\) is a homogeneous simplicial complex of codimension \(1\); 
\eqref{eq:Alhforslike} is reminiscent of an \emph{Alhfors upper codimension-\(1\) bound} (see \cites{Maly2017,Saksman_Soto_2017,Maly_Shanmugalingam_Snipes_2018}, where a doubling condition is made separately). The quantity \(\seminorm{\sigma}_{\mathrm{Lip}}\) denotes the Lipschitz constant of the map \(\sigma\colon \Sigma \to \R^m_+\)
\[
 \seminorm{\sigma}_{\mathrm{Lip}} \defeq 
 \sup_{x, y \in \Sigma} \frac{\abs{\sigma (x) - \sigma (y)}} {d_{\Sigma} (x, y)}.
\]
We define the \emph{canonical cubication of the half-space} \(\R^m_+\) of size \(\kappa\) as follows. 
For \(\ell \in \{0, \dotsc, m\}\) and \(\kappa > 0\), we write 
\begin{equation}\label{eq:cubicationdefintion1}
 \mathcal{Q}^{\kappa, \ell}
 \defeq \set*{Q \subset \R^m \colon Q \text{ is a closed $\ell$--dimensional face of } \left[-\frac{\kappa}{2}, \frac{\kappa}{2}\right]^m + \kappa k \text{ with } k \in \Z^m}
\end{equation}
and we let then 
\begin{align}\label{eq:cubicationdefintion2}
 \mathcal{Q}^{\kappa, \ell}_+ &\defeq \set[\big]{Q \cap \oline{\R}^{m}_+ \colon Q \in \mathcal Q^{\kappa, \ell} },
 &
\mathcal Q^{\kappa, \ell}_0 & \defeq \set[\big]{Q \cap \partial \R^{m}_+ \colon Q \in \mathcal Q^{\kappa, \ell + 1} }
\end{align}
(with \(\ell \in \{0, \dotsc, m- 1\}\) in the definition of \(\mathcal Q^{\kappa, \ell}_0\); since the faces of cubes of \(\mathcal{Q}^{\kappa, \ell}\) all intersect transversally \(\partial \R^m_+\), the cubes in \(\mathcal{Q}^{\kappa, \ell}_0\) are \(\ell\)--dimensional) and 
\begin{align}\label{eq:cubicationdefintion3}
 \mathcal{C}^{\kappa, \ell}_+ &= \bigcup \mathcal Q^{\kappa, \ell}_+, &
\quad \mathcal{C}^{\kappa, \ell}_0 &= \bigcup \mathcal Q^{\kappa, \ell}_0.
\end{align}

We will state our results for a mapping \(u\) from \(\partial \R^{m}_+\) \emph{defined everywhere} following \citelist{\cite{White-homotopy}*{p.\ 5}\cite{Hang_Lin_II}*{p.\ 66}}. In other words, we do not consider equivalence classes of functions equal almost everywhere. (Otherwise, given any \(\sigma\) or any \(h\), there exists a map equal almost everywhere that satisfies  \ref{it_aithesh6rea9ouGaib0Dae2f} or \ref{it_ap6kahs3enainah2lohz9ET2} in \Cref{theorem_main_halfspace} by being constant on the set \(\sigma(\Sigma_0)\) or on \(\mathcal{C}^{\kappa_i, \floor{p - 1}}_0 + h \).)

We obtain the following characterization of traces.
\begin{theorem}
\label{theorem_main_halfspace}
Assume that  $2 \le m\in\N$, the map $u\colon \partial \R^{m}_+ \to \n$ is Borel--measurable, and \(\lambda > 1\). 
If \(1 < p < m\), then the following statements are equivalent:
\begin{enumerate}[label=(\roman*)]
 \item \label{it_eiquoongeiJ8quahb5yu7hoo}
 There exists \(U \in \dot{W}^{1, p} (\R^{m}_+, \n)\) such that $u = \tr_{\partial \R^{m}_+} U$.
 \item \label{it_aithesh6rea9ouGaib0Dae2f} There exists a summable Borel--measurable function \(w \colon \partial \R^{m}_+ \to [0,\infty]\)
such that for every finite homogeneous simplicial complex $\Sigma$ of dimension \(\floor{p}\) and every subcomplex $\Sigma_0\subset \Sigma$ of dimension \(\floor{p - 1}\), and every Lipschitz--continuous map $\sigma\colon \Sigma \to \R^{m}_+$ which satisfies \(\sigma (\Sigma_0) \subset \partial \R^m_+\) and for which \( w \compose \sigma\) is summable, 
there exists a mapping $V \in \dot{W}^{1, p} (\Sigma, \n)$
 such that $\tr_{\Sigma_0} V  =  u \compose \sigma \big\vert_{\Sigma_0}$ and
 \[
  \int_{\Sigma} |\deriv V|^p \leq
  \gamma^\lambda_{\Sigma_0, \Sigma} \,\seminorm{\sigma}_{\lip}^{p - 1}
  \int_{\Sigma_0} w\compose \sigma.
 \]
  \item 
  \label{it_ap6kahs3enainah2lohz9ET2}
 There exist a constant \(\theta > 0\), a sequence \((\kappa_i)_{i \in \N}\) in \((0, \infty)\) converging to \(0\), and sets \(H_i \subset [0, \kappa_i]^{m - 1}\), such that 
 \[
  \liminf_{i \to \infty} \frac{\abs{H_i}}{\kappa_i^{m - 1}} > 0 
 \]
and if \(h \in H_i \times \{0\}\), then there exists a map \(V \in \dot{W}^{1, p} (\mathcal{C}^{\kappa_i, \floor{p}}_+ + h, \n)\) such that \(\tr V = u \big\vert_{\mathcal{C}^{\kappa_i, \floor{p - 1}}_0 + h}\) and 
 \[
 \kappa_i^{m - \floor{p}}
   \smashoperator{\int_{\mathcal{C}^{\kappa_i, \floor{p}}_+ + h}} \abs{\deriv V}^p \le  \theta.
 \]
\end{enumerate}
\end{theorem}

The function \(w\) appearing in \ref{it_aithesh6rea9ouGaib0Dae2f} can be interpreted as \emph{an extension energy density}; the mappings \(\sigma\) can be interpreted as generalized paths going through \(\oline{\R}^m_+\).

In the paths condition \ref{it_aithesh6rea9ouGaib0Dae2f}, we emphasize the facts that, as in singular homology, we do not assume anything about the local or global injectivity of \(\sigma\) --- the map \(\sigma\) could even take a constant value where \(w\) is finite, in which case \(u \compose \sigma \big\vert_{\Sigma_0}\) is of course trivially extended by a constant --- and that there is no Jacobian appearing in \(\int_{\Sigma_0} w \compose \sigma\): we are integrating \(w \compose \sigma\) on \(\Sigma_0\) rather than integrating \(w\) on the set \(\sigma (\Sigma_0)\).

Assertion \ref{it_ap6kahs3enainah2lohz9ET2} is very rigid because of the presence of a cubication, whereas assertion \ref{it_aithesh6rea9ouGaib0Dae2f} is very robust --- it is invariant under diffeomorphisms whose derivative and its inverse are controlled uniformly --- and is thus a natural candidate for a geometrical characterization of the image of the trace operator.

In broad terms, the proof of \cref{theorem_main_halfspace} consists in deducing \ref{it_aithesh6rea9ouGaib0Dae2f} from \ref{it_eiquoongeiJ8quahb5yu7hoo} by a Fubini type argument (the proof is given in \cref{ss:1implies2}), \ref{it_ap6kahs3enainah2lohz9ET2} from \ref{it_aithesh6rea9ouGaib0Dae2f} by the particularization to families of translations of canonical cubical complexes (the proof is given in \cref{ss:2implies3}), and \ref{it_eiquoongeiJ8quahb5yu7hoo} from \ref{it_ap6kahs3enainah2lohz9ET2} by defining homogeneous extensions on cubical skeletons (the proof is given in \cref{ss:3implies1}).

We next have a geometric statement of \cref{theorem_main_halfspace} on manifolds.

\begin{theorem}
\label{theorem_global_necessary_sufficient}
Let \(\M\) be an \(m\)--dimensional compact Riemannian manifold with boundary \(\partial \M\) and \(1 < p < m\). There exists a \(\delta > 0\) and a $\lambda>1$ such that the following conditions are equivalent:

\begin{enumerate}[label=(\roman*)]
    \item For any Borel--measurable map \(u \colon  \partial \M \to \n\), 
there exists a map \(U \in \dot{W}^{1, p} (\M, \n)\) with \(\tr_{\partial \M} U = u\).
    
    \item There exists a summable Borel--measurable function 
\(w \colon  \partial \M \to [0,\infty]\)
such that for every homogeneous simplicial complex \(\Sigma\) of dimension \(\floor{p}\), for every subcomplex \(\Sigma_0\subset \Sigma\) of dimension \(\floor{p - 1}\),  and for every Lipschitz--continuous map \(\sigma \colon  \Sigma \to \M\) satisfying 
\begin{align*}
 \sigma (\Sigma_0) &\subseteq \partial \M, 
 &
 \seminorm{\sigma}_{\lip} \sup_{\Sigma} d_0 & \le \delta, &
 & \text{ and }&
 \int_{\Sigma_0} w \compose \sigma &< \infty,
\end{align*}
 there exists a mapping \(V \in \dot{W}^{1, p} (\Sigma, \n)\) such that \(\tr_{\Sigma_0} V   = u \compose \sigma \big \vert_{\Sigma_0}\) and
\begin{equation}\label{eq:againthesame}
 \int_{\Sigma} \abs{\deriv V}^p 
 \leq
  \gamma^\lambda_{\Sigma_0, \Sigma}
   \,
   \seminorm{\sigma}_{\lip}^{p - 1} \int_{\Sigma_0} w \compose \sigma.
\end{equation}
\end{enumerate}
\end{theorem}

Here we have defined $d_0\colon \Sigma \to \R$ to be the distance to \(\Sigma_0\) in \(\Sigma\) by
\begin{equation}
\label{eq_phuYoo0euhoi7shu5Saeng9}
 d_0 (y) \defeq \inf \{d^\Sigma (y, z) \colon z \in {\Sigma_0}\};
\end{equation}
the quantity \(\sup_{\Sigma} d_0\) quantifies how far points in \(\Sigma\) can be from \(\Sigma_0\).

In comparison with \cref{theorem_main_halfspace}, the map \(\sigma\) is assumed to satisfy the nonlinear conditions that \(\sigma (\Sigma) \subseteq \m\) and \(\sigma(\Sigma_0) \subseteq \partial \m\). 
The proof of \cref{theorem_global_necessary_sufficient} is based on the proof of \cref{theorem_main_halfspace} through suitable localization arguments.

Finally, as in Isobe's characterization by \ref{it_zuugh7Thai5ThiephahtoP8y} and \ref{it_dahGaiquu0xejai3Tah8aipa}, the obstruction to the extension can be decoupled into a quantitative obstruction to the extension to a neighborhood of the boundary and a qualitative obstruction to the extension to the whole manifold.

\begin{theorem}
\label{proof_global_decomposition}
Let \(\M\) be an \(m\)--dimensional compact Riemannian manifold with boundary \(\partial \M\) and let \(1 < p <m\).
There exists a \(\delta > 0\) and a $\lambda>1$ such that for each \(u\colon \partial \m \to \n\)
 the existence of an extension \(U \in \dot{W}^{1, p} (\M, \n)\) with \(\tr_{\partial \M} U = u\) is equivalent to the following property:
 
 There exists a summable function \(w \colon \partial \M \to [0,\infty]\) such that for every finite homogeneous simplicial complex \(\Sigma\) of dimension $\floor p$, every subcomplex \(\Sigma_0\subset \Sigma\) of codimension \(1\), and every Lipschitz--continuous mapping \(\sigma \colon \Sigma \to \M\) satisfying \(\int_{\Sigma_0} w \compose \sigma < \infty\), one has:
\begin{enumerate}
[label=(\alph*)]
\item 
 \label{it_kuixoongeeleag7aeVeiyie4}
 If \(\sigma (\Sigma) \subseteq \partial \M\) and \(\seminorm{\sigma}_{\lip}\sup_{\Sigma}d_0 \le \delta\), then there exists a mapping \(V \in \dot{W}^{1, p} (\Sigma, \n)\) such that \(\tr_{\Sigma_0} V = u \compose \sigma\big\rvert_{\Sigma_0}\) and 
\[
 \int_{\Sigma} \abs{\deriv V}^p 
 \le \gamma^\lambda_{\Sigma_0, \Sigma}\, \seminorm{\sigma}_{\lip}^{p - 1} \int_{\Sigma_0} w \compose \sigma.
\]
\item
\label{it_Chea0siechoa5uTieca2Xooc} If \(\sigma (\Sigma_0) \subseteq \partial \M\), then \(u \compose \sigma \big\vert_{\Sigma_0}\) is homotopic in \(\mathrm{VMO} (\Sigma_0, \n)\) to \(V \big\vert_{\Sigma_0}\) for some  \(V \in C (\Sigma, \n)\).
\end{enumerate}
\end{theorem}

The assertion \ref{it_kuixoongeeleag7aeVeiyie4} differs from the condition of \cref{theorem_global_necessary_sufficient}, by the fact that in \ref{it_kuixoongeeleag7aeVeiyie4} we assume the stronger condition that \(\sigma (\Sigma) \subseteq \partial \m\) instead of the weaker condition that \(\sigma (\Sigma_0) \subseteq \partial \m\) and  \(\sigma (\Sigma) \subseteq \m\), resulting in a weaker condition; in order to keep the equivalence we supplement \ref{it_kuixoongeeleag7aeVeiyie4} with \ref{it_Chea0siechoa5uTieca2Xooc}, which is a reformulation of Isobe's condition \ref{it_dahGaiquu0xejai3Tah8aipa} as a condition on paths, as it appears in topological screening for the approximation of Sobolev mappings \cite{BPVS}.
In the particular case where \(p \not \in \N\), assertion \ref{it_Chea0siechoa5uTieca2Xooc} is equivalent to the fact that \(u \compose \sigma \big\vert_{\Sigma_0}\) is almost everywhere equal to the restriction to \(\Sigma_0\) of some \(V \in C (\Sigma, \n)\).
In contrast with \cref{theorem_global_necessary_sufficient}, \cref{proof_global_decomposition} does not give a quantitative estimate; such an estimate is precluded by the qualitative character of assertion \ref{it_Chea0siechoa5uTieca2Xooc}.

\subsection*{Acknowledgments}
K.M. was supported by FSR Incoming postdoc.
K.M. and J.V.S. were both supported by the Mandat d'Impulsion Scientifique F.4523.17, ``Topological singularities of Sobolev maps'' of the Fonds de la Recherche Scientifique--FNRS.

\section{The case of half-spaces}\label{s:halfspace}

\subsection{From the extension to the paths condition}
\label{ss:1implies2}
We prove that \ref{it_eiquoongeiJ8quahb5yu7hoo} implies \ref{it_aithesh6rea9ouGaib0Dae2f} in \cref{theorem_main_halfspace}.

\begin{proposition}
\label{pr:trace_trails_necessary_2}
Let \(2 \le m \in \N\), \(\lambda > 1\), and \(p\in [1,\infty)\).
There exists a constant \(C\) such that given any \(U \in \dot{W}^{1, p} (\R^m_+,\n)\) there exists a Borel--measurable function \(w \colon \partial \R^{m}_+ \to [0, \infty]\) with
\[
 \int_{\partial \R^{m}_+} w \le C \int_{\R^m_+} \abs{\deriv U}^p
\]
with the following property:

Suppose that $\Sigma$ is a finite homogeneous simplicial complex, $\Sigma_0 \subset \Sigma$ is a subcomplex of codimension 1, that the map \(\sigma \colon \Sigma \to \oline{\R}^m_+\) is Lipschitz--continuous and satisfies  \(\sigma (\Sigma_0) \subset \partial \R^m_+\), and that
\[
\int_{\Sigma_0} w \compose \sigma < \infty. 
\]
Then there exists a map \(V \in  \dot{W}^{1, p} (\Sigma, \n)\) with \(\tr_{\Sigma_0}V= U \compose \sigma \big \vert_{\Sigma_0}\) almost everywhere on \(\Sigma_0\) and 
\begin{equation}\label{eq:estimateofUonskeleton}
  \int_{\Sigma} \abs{\deriv V}^p\\
  \le 
  \gamma^{\lambda}_{\Sigma_0, \Sigma} \,\seminorm{\sigma}_\lip^{p-1}
  \int_{\Sigma_0}  w \compose \sigma,
\end{equation}
where $\gamma^{\lambda}_{\Sigma_0, \Sigma}$ is defined as in \eqref{eq:Alhforslike}.
\end{proposition}

The conclusion of \cref{theorem_main_halfspace} where the complex \(\Sigma\) has arbitrary -dimension is slightly stronger than \ref{it_aithesh6rea9ouGaib0Dae2f} in \cref{theorem_main_halfspace} where the dimension of \(\Sigma\) is \(\floor{p}\).

Here and in the sequel, for $0<\eta<1$ and \(\rho > 0\) we define the \emph{solid spherical cap}
\begin{equation}\label{eq:solidsphericalcup}
 C^\eta_\rho \defeq  \{x = (x', x_m) \in\R^m\colon \abs{x}< \rho \text{ and } x_m > \eta \rho\}
 = B_\rho \cap (\R^{m - 1} \times (\eta \rho, \rho))
\end{equation}
and note that 
\[
 \abs{C_\rho^{\eta}} = \rho^m \abs{C_1^\eta}.
\]
The main ingredient of the proof is the following integration lemma.
\begin{lemma}
\label{le:integration_2}
Let $0<\eta<1$, $\lambda>\frac{1}{\eta}$, and $\rho>0$. Assume that \(F \colon \R^m_+ \to [0, \infty]\) is a Borel--measurable function, \(\Sigma\) is a finite simplicial complex, and \({\Sigma_0} \subset \Sigma\) is a homogeneous subcomplex of codimension one.  Assume moreover that \(\sigma \colon \Sigma \to \oline{\R}^m_+\) is a Lipschitz--continuous map with $\sigma(\Sigma_0)\subseteq \partial \R^m_+$ and  \(\seminorm{\sigma}_{\mathrm{Lip}}< (\eta \lambda - 1) \rho \).
Then we have
\begin{equation}
\label{eq_aijalanguGh8oochoh2wuel6}
\begin{split}
  \int\limits_{C^{\eta}_{\rho}} \biggl(\int\limits_{\Sigma} F\compose (\sigma +  d_0\xi ) &\biggr)\dif \xi\\
  &\le \frac{(\eta \lambda - 1)^m(\rho + \seminorm{\sigma}_{\mathrm{Lip}})^{2 m}}{\eta^m \brk*{(\eta \lambda  - 1) \rho - \seminorm{\sigma}_{\mathrm{Lip}}}^{m + 1}}
  \gamma^{\lambda}_{\Sigma_0, \Sigma}
  \int\limits_{{\Sigma_0}} \int\limits_{\R^m_+} \frac{F (x) x_m}{\abs{x - \sigma (z)}^m}\dif x \dif z.
\end{split}
  \end{equation}
\end{lemma}

We recall that the quantity $\gamma^{\lambda}_{\Sigma_0, \Sigma}$ was defined in \eqref{eq:Alhforslike}; here and in the sequel we write \((x', x_m) = x \in \R^m_+ = \R^{m - 1} \times (0,\infty)\).

\begin{proof}%
[Proof of \cref{le:integration_2}]
By a change of variables
\(x = \sigma(y) +  d_0 (y)\xi\), with \(d_0\) having been defined in \eqref{eq_phuYoo0euhoi7shu5Saeng9}, and in view of \eqref{eq:solidsphericalcup} and of the non-negativity of the last component \(\sigma_m\) of \(\sigma \colon \Sigma \to \oline{\R}^m_+\), we have 
\begin{equation}\label{eq_oht6ooquu6caghooB9thahru}
\begin{split}
 \int_{C^{\eta}_\rho} \brk*{\int_{\Sigma} F\compose \brk*{\sigma +  d_0 \xi}}\dif \xi
 &= \int\limits_{\Sigma} \brk[\Bigg]{\hspace{.3em}\smashoperator[r]{\int\limits_{\substack{\abs{x - \sigma (y)} \le \rho  d_0 (y)\\ x_m \ge \rho\eta d_0 (y) + \sigma_m (y)}}} \frac{F (x)}{d_0 (y)^m}  \dif x }\dif y\\
 &\le \int\limits_{\Sigma} \brk[\Bigg]{\hspace{.3em}\smashoperator[r]{\int\limits_{\substack{\abs{x - \sigma (y)} \le \rho  d_0 (y)\\ x_m \ge \eta \rho d_0 (y) }} } \frac{F (x)}{d_0 (y)^m}  \dif x }\dif y\\
 &= \int\limits_{\Sigma} 
 \fint\limits_{{\Sigma_0} \cap B^\Sigma_{\tau d_0 (y)} (y)}  \Biggl(\hspace{.3em}\smashoperator[r]{\int\limits_{\substack{\abs{x - \sigma (y)} \le \rho  d_0 (y)\\ x_m \ge \eta \rho d_0 (y)}}}
 \frac{F (x)}{d_0 (y)^m}  \dif x \Biggr)\dif z \dif y,
\end{split}
\end{equation}
where $\tau>1$ is to be chosen later (see \eqref{eq_saubee2ohcoo2Uikiav0JooP} below). 

Noting that for every \(x \in \R^m_+\) satisfying $|x-\sigma(y)|\le \rho d_0(y)$ and \(z\in {\Sigma_0} \cap B^\Sigma_{\tau d_0 (y)} (y)\) we have
\[
\begin{split}
\abs{x - \sigma (z)} &\le \abs{x - \sigma (y)} + \abs{\sigma(y) - \sigma(z)} \\
&\le \abs{x - \sigma(y)} + \seminorm{\sigma}_{\mathrm{Lip}} d^\Sigma (z, y)
\le (\rho  + \seminorm{\sigma}_{\mathrm{Lip}}\tau ) d_0 (y), 
\end{split}
\]
and since, by assumption, \(\sigma_m (y) \ge 0\) and $\sigma(\Sigma_0)\subseteq \partial \R^m_+$ we also have
\[
 x_m \le \abs{x_m - \sigma_m (y)} + \sigma_m (y)
 \le (\rho + \seminorm{\sigma}_{\mathrm{Lip}}) d_0 (y). 
\]
Thus, we estimate by \eqref{eq_oht6ooquu6caghooB9thahru}
\begin{equation}
\label{eq:Festimateone}
\begin{split}
  \int_{C^\eta_\rho}   &\brk*{\int_{\Sigma} F\compose (\sigma + d_0\xi )}\dif \xi\\
  &\le \brk*{\rho + \seminorm{\sigma}_{\mathrm{Lip}}}^m \int\limits_{\Sigma_0} \smashoperator[r]{\int\limits_{\abs{x - \sigma (z)} \le \frac{\rho + \seminorm{\sigma}_{\mathrm{Lip}}\tau}{\kappa} x_m}} \quad \frac{F (x)}{x_m^m}
  \brac{\int\limits_{S_{x,z}} \frac{1}{\abs*{{\Sigma_0} \cap B^\Sigma_{\tau d_0 (y)} (y)}} \dif y  } \dif x \dif z,
\end{split}
\end{equation}
with the set \(S_{x,z} \subseteq \Sigma\) being defined as 
\begin{equation*}
S_{x, z} \defeq \left\{y \in \Sigma : d^\Sigma(y, z) \le \tau d_0 (y) \text{ and } \frac{x_m}{K} \le d_0 (y) \le \frac{x_m}{\kappa}\right\}
\end{equation*}
and
\begin{align}
\label{eq_hahpheiwoukueph3ii7yoPhi}
 K  &\defeq \rho + \seminorm{\sigma}_{\mathrm{Lip}}, &
 \kappa &\defeq  \eta \rho.
\end{align}

We estimate now the innermost integral of the right-hand side of \eqref{eq:Festimateone}.
Since the set \(\Sigma_0\) is compact, for every \(y \in \Sigma\), there exists a point \(y_0 \in {\Sigma_0}\) such that \(d_0 (y) = d^{\Sigma} (y, y_0)\). Thus, if \(d_0 (y) \ge x_m/K\), we have
\begin{equation}
\label{eq_puuleek4IemaeFi9oh0Iewao}
 \abs[\big]{{\Sigma_0} \cap B^\Sigma_{\tau d_0 (y)} (y)                                      } \ge \abs[\big]{{\Sigma_0} \cap B^\Sigma_{(\tau - 1) d_0 (y)} (y_0)}
 \ge \abs[\big]{{\Sigma_0} \cap B^\Sigma_{(\tau - 1) x_m/K} (y_0)}.
\end{equation}
Moreover, 
\begin{equation}
\label{eq_Xiaquob6vaen8gei8ek6yohm}
 S_{x, z} \subseteq B^{\Sigma}_{(\tau x_m)/\kappa} (z).
\end{equation}
It follows thus from \eqref{eq_puuleek4IemaeFi9oh0Iewao} and \eqref{eq_Xiaquob6vaen8gei8ek6yohm} that for every \(x \in \R^m_+\) and \(z \in {\Sigma_0}\)
\begin{equation}\label{eq:innerest}
\begin{split}
\int\limits_{S_{x, z}} \frac{\dif y }{\abs{\Sigma_0 \cap B^\Sigma_{\tau d_0 (y)} (y)) }} 
&\le \frac{(\tau - 1)x_m}{K}\sup_{y_0 \in {\Sigma_0}} \frac{\abs{B^\Sigma_{(\tau x_m)/\kappa} (y_0)} }
{\frac{(\tau - 1) x_m }{K} \abs{{\Sigma_0} \cap B^\Sigma_{(\tau - 1)x_m/K} (y_0)}}\\
&\le \frac{(\tau-1) x_m}{K} \sup_{\substack{y_0\in{\Sigma_0} \\ \delta>0} } \frac{\big|B^\Sigma_{\frac{\tau}{\tau -1} \frac{K}{\kappa}\delta}(y_0)\big|}{\delta \abs{{\Sigma_0} \cap B^\Sigma_{\delta}(y_0)}}.
\end{split}
\end{equation}
Recalling that by assumption $(\eta \lambda   - 1) \rho - \seminorm{\sigma}_{\mathrm{Lip}}>0$, we set
\begin{equation}
\label{eq_saubee2ohcoo2Uikiav0JooP}
 \tau \defeq \frac{\eta \lambda \rho}{(\eta \lambda   - 1) \rho - \seminorm{\sigma}_{\mathrm{Lip}}},
\end{equation}
so that one can directly check that in view of \eqref{eq_hahpheiwoukueph3ii7yoPhi}
\begin{align*}
 \frac{\tau}{\tau -1} \frac{K}{\kappa} &= \lambda, 
 &\frac{\tau -1}{K} &= \frac{1}{(\eta \lambda   - 1) \rho - \seminorm{\sigma}_{\mathrm{Lip}}},&
 \frac{\rho + \seminorm{\sigma}_{\mathrm{Lip}}\tau}{\kappa} &= \frac{(\rho + \seminorm{\sigma}_{\mathrm{Lip}})(\eta \lambda - 1)}{\eta ((\eta \lambda  - 1) \rho - \seminorm{\sigma}_{\mathrm{Lip}})}.
\end{align*}
Moreover, setting $\theta \defeq \frac{\rho + \seminorm{\sigma}_{\mathrm{Lip}}\tau}{\kappa} = \frac{(\eta \lambda  - 1)(\rho + \seminorm{\sigma}_{\mathrm{Lip}})}{\eta ((\eta \lambda  - 1) \rho - \seminorm{\sigma}_{\mathrm{Lip}})}$ we get
\begin{equation}
\label{eq:integralFoverxm}
\int\limits_{\substack{x \in \R^m_+\\ 
\abs{x - \sigma(z)} \le \theta x_m}} \frac{F (x)}{x_m^{m - 1}} \dif x
\le \theta^m \int_{\R^m_+} \frac{x_m F (x)}{\abs{x - \sigma(z)}^m}\dif x.
\end{equation} 
Combining \eqref{eq:Festimateone} with \eqref{eq:innerest} and \eqref{eq:integralFoverxm} we conclude. 
\end{proof}

\begin{proof}%
[Proof of \cref{pr:trace_trails_necessary_2}]%
\resetconstant
Without loss of generality, we assume that 
\[
 \int_{\R^m_+} \vert \deriv U\vert^p > 0.
\]
We consider a sequence \((U_j)_{j \in \N}\) in \(C^\infty (\oline{\R}^m_+, \R^\nu)\) such that for each \(j \in \N\)
\begin{gather}
 \label{eq_shoo2EZeepuquureeKahr4So} \int_{\R^m_+} \abs{\deriv U_j - \deriv U}^p \le \frac{1}{2^j},\\
 \label{eq_ahThaiMie7vohru0phai2Aa3}
 \int_{B^m_j (0)} \abs{U_j - U}^p \le \frac{1}{2^j},\\
 \label{eq:Ujconvergenceonboundary}
\int_{B^{m - 1}_j (0)} \abs{U_j - u}^p \le \frac{1}{2^j},
\end{gather}
where $u\coloneqq \tr_{\partial \R^m_+} U$. Defining the function \(W \colon \R^m_+ \to [0, \infty]\) by 
\begin{equation}
\label{eq_wiuw3je7aroo3eeG9vaez4ah}
W \defeq \abs{\deriv U}^p
+ 
\brk*{\sum_{j \in \N} \abs{\deriv U_j - \deriv U}^p + \sum_{j \in \N} \chi_{\R^m_+\cap B^m_j(0)}\abs{U_j - U}^p}\int_{\R^m_+} \abs{\deriv U}^p ,
\end{equation}
we have by \eqref{eq_wiuw3je7aroo3eeG9vaez4ah}, \eqref{eq_shoo2EZeepuquureeKahr4So}, and \eqref{eq_ahThaiMie7vohru0phai2Aa3}
\begin{equation}
\label{eq_Shieh7eef4vah6TadaYiekah}
\begin{split}
 \int_{\R^m_+} W &= \int_{\R^m_+} \abs{\deriv U}^p \brk*{1 + \sum_{j \in \N} \int_{\R^m_+}\brac{|\deriv U_j - \deriv U|^p + \chi_{\R^m_+ \cap B^m_j(0)}|U_j - U|^p} } \\
 & \le  \int_{\R^m_+} \abs{\deriv U}^p \brk*{1 + 2\sum_{j \in \N} \frac{1}{2^{j}}} = 5 \int_{\R^m_+} \abs{\deriv U}^p.\ 
\end{split}
\end{equation}
We define the function \(w \colon \partial \R^{m}_+ \to [0,\infty]\) for each $y\in \partial \R^{m}_+$ by
\begin{equation}
\label{eq:defH}
 w(y) 
 \defeq 
 \int_{\R^m_+} \frac{W (x)x_m}{\abs{x - y}^m}\dif x +  \sum_{j \in \N} \chi_{B^{m - 1}_j (0)}\abs {U_j - u}^p \int_{\R^m_+} \abs{\deriv U}^p.
\end{equation}
We have by \eqref{eq_Shieh7eef4vah6TadaYiekah} and \eqref{eq:defH}
\begin{equation}
\label{eq_tha1Sa8eitheis1lash3wi2b}
\begin{split}
  \int_{\partial \R^{m}_+} \int_{\R^m_+} \frac{W (x)x_m}{\abs{x - y}^m} \dif x \dif y
 &= \int_{\R^m_+} W (x) \int_{\partial \R^{m}_+} \frac{x_m}{\abs{x - y}^m} \dif y \dif x \\
&= \int_{\partial \R^{m}_+} \frac{1}{(\abs{y}^2 + 1)^\frac m2} \dif y
\int_{\R^m_+} W \le \C \int_{\R^m_+} \abs{\deriv U}^p.
\end{split}
\end{equation}
We also have by 
\eqref{eq:Ujconvergenceonboundary}
\begin{equation}
\label{eq_imoo2ieVie3Is2vui7Cahriu}
 \sum_{j \in \N} \int_{B^{m - 1}_j (0)}\abs{U_j - u}^p
 < \infty,
\end{equation}
so that in view of \eqref{eq:defH},  \eqref{eq_tha1Sa8eitheis1lash3wi2b}, and \eqref{eq_imoo2ieVie3Is2vui7Cahriu}
\begin{equation*}
\label{eq:Hintegrable}
 \int_{\partial \R^{m}_+} w
 \le \C \int_{\R^m_+} \abs{\deriv U}^p <\infty.
\end{equation*}

We assume now that \(\sigma \colon \Sigma \to \oline{\R}^m_+\) is Lipschitz--continuous, that \(\sigma(\Sigma_0)\subset \partial \R^m_+\), and that 
\[
 \int_{\Sigma_0} w \compose\sigma  < \infty.
\]
For each \(\xi \in C^{\eta}_\rho\), we define the map
\begin{equation}
\label{eq_aiqu4quohn9aeSheeghoeshi}
 \sigma_{\xi}  \defeq \sigma  + d_0 \xi
 \colon \Sigma \to \oline{\R}^m_+.
\end{equation}
We claim that the map \(V \defeq U \compose \sigma_\xi\) satisfies the conclusion for a suitable \(\xi \in C^{\eta}_\rho\).

Indeed, by Lemma~\ref{le:integration_2} with 
\begin{equation}\label{eq:choiceofrho}
\rho = 2\seminorm{\sigma}_\lip/(\eta \lambda - 1),
\end{equation}
we first obtain 
\begin{equation*}
  \int_{C^{\eta}_\rho} \int_{\Sigma} W \compose \sigma_\xi \dif \mu \dif \xi\\
  \\
  \le 
 \frac{(\eta \lambda + 1)^{2 m}}{\eta^m(\eta \lambda - 1)^m}
  \seminorm{\sigma}_\lip^{m - 1}\,
  \gamma^{\lambda}_{\Sigma_0, \Sigma}
  \int\limits_{\Sigma_0} w \compose \sigma.
 \end{equation*}
Since $|C_\rho^\eta| = \rho^m |C_1^\eta|$,  there exists a \(\xi \in C_\rho^{\eta}\) such that
\begin{equation}
\label{eq_lae5shaevao3eepheikee5Ce}
\int_{\Sigma} W \compose \sigma_\xi \dif \mu
\le   \frac{(\eta \lambda + 1)^{2 m}}{2^m \eta^m |C_1^\eta|}
   \frac{\gamma^{\lambda}_{\Sigma_0, \Sigma}}{\seminorm{\sigma}_\lip}
  \int\limits_{\Sigma_0} w \compose \sigma;
\end{equation}
we fix such a \(\xi\) for the remainder of the proof.
Since we have assumed that \(\int_{\R^m_+} \abs{\deriv U}^p > 0\), \eqref{eq_lae5shaevao3eepheikee5Ce} implies in view of \eqref{eq_wiuw3je7aroo3eeG9vaez4ah} that 
\begin{multline}
\label{eq_id2phiga6jaefiereiCheiR8}
 \sum_{j \in\N} \int_\Sigma \abs{\deriv U_j(\sigma_\xi (y)) - \deriv U(\sigma_\xi (y))}^p \dif y\\
 + \int_{\Sigma}\chi_{\R^m_+\cap B^m_j(0)} (\sigma_\xi(y)) \abs{U_j(\sigma_\xi (y)) - U(\sigma_\xi (y))}^p \dif y <\infty.
\end{multline}
By Lipschitz--continuity of $\sigma_\xi$ and smoothness of $U_j$, we have \(\deriv (U_j \compose \sigma_\xi) = \deriv U_j (\sigma_\xi) \cdot \deriv \sigma_\xi\) almost everywhere (here, by $\cdot$ we mean the composition of differential as linear mappings, or equivalently, the multiplication of the Jacobian matrices) and thus by \eqref{eq_id2phiga6jaefiereiCheiR8}
\begin{multline}
\label{eq_phooriutahw2Ruuxuoyohd8e}
\sum_{j \in\N} \int_\Sigma \abs{\deriv (U_j \compose \sigma_\xi) - \deriv U (\sigma_\xi) \cdot \deriv \sigma_\xi}^p \dif y \\
\le 
\sum_{j \in\N} \int_\Sigma \abs{\deriv U_j(\sigma_\xi(y)) - \deriv U (\sigma_\xi(y))}^p\abs{\deriv \sigma_\xi}^p \dif y 
< \infty.
\end{multline}
Thus, by \eqref{eq_id2phiga6jaefiereiCheiR8} and \eqref{eq_phooriutahw2Ruuxuoyohd8e}, we have 
\[
 \lim_{j\to \infty} \int_{\Sigma} \abs{U_j \compose \sigma_\xi - U\compose \sigma_\xi}^p +
 \int_{\Sigma}\abs{\deriv \brac{U_j \compose \sigma_\xi} - \deriv U(\sigma_\xi) \cdot \deriv \sigma_\xi}^p = 0, 
\]
this implies that $U\compose \sigma_\xi \in \dot{W}^{1,p}(\Sigma,\n)$ and $\deriv (U \compose \sigma_\xi) = \deriv U (\sigma_\xi)\cdot \deriv \sigma_\xi$. Finally, since by definition of \(w\) we have
\[
 \sum_{j \in \N} \int_{\Sigma_0} (\chi_{B^{m - 1}_j (0)} \compose \sigma_\xi) \abs{U_j \compose \sigma_\xi - u \compose \sigma_\xi}^p < \infty,
\]
and hence \(U_j \compose \sigma_\xi \to u \compose \sigma_\xi\) in $L^p_{\mathrm{loc}} (\Sigma_0)$. By continuity of the trace, \(\tr_{\Sigma_0} U \compose \sigma_\xi = u \circ \sigma_\xi\big\vert_{\Sigma_0}\).

In order to conclude, we note that since $\rho =2\seminorm{\sigma}_\lip/(\eta \lambda - 1)$ and  \(\xi\in C_\rho^\eta\), we have
\begin{align*}
 |\deriv \sigma_\xi| &\le \abs{\deriv \sigma} + \abs{\xi} &
 &\text{ and} &
 |\xi| &\le \frac{2\seminorm{\sigma}_\lip}{(\eta \lambda - 1)}.
\end{align*}
Thus,
\[
 \int_{\Sigma }|\deriv (U \compose \sigma_\xi)|^p   \le |\sigma_\xi|_\lip^p \int_{\Sigma}|\deriv U \compose \sigma_\xi|^p
 \le \frac{\seminorm{\sigma}^p_{\lip}(\eta \lambda + 1)^p}{(\eta \lambda - 1)^p} \int_{\Sigma}|\deriv U \compose \sigma_\xi|^p,
\]
which gives, by \eqref{eq_lae5shaevao3eepheikee5Ce},
\[
\int_{\Sigma }|\deriv (U \compose \sigma_\xi)|^p 
\le 
 \frac{(\eta \lambda + 1)^{2 m+p}}{2^m \eta^m (\eta\lambda-1)^p |C_1^\eta|}
  \seminorm{\sigma}_\lip^{p - 1} \gamma^{\lambda}_{\Sigma_0, \Sigma}
 \int\limits_{\Sigma_0} w \compose \sigma.
\]
We take \(\eta \defeq \frac{\lambda + 1}{2\lambda}\) and multiply \(w\) by a suitable constant, so that \eqref{eq:estimateofUonskeleton} holds.
\end{proof}

\subsection{From paths to cubical meshes}\label{ss:2implies3}

The implication \ref{it_aithesh6rea9ouGaib0Dae2f} \(\implies \) \ref{it_ap6kahs3enainah2lohz9ET2} in \cref{theorem_main_halfspace} will follow from the next proposition.

\begin{proposition}
\label{pr:(2)to(3)}
Let \(2 \le m \in \N\), \(\lambda > 1\), \(p\in [1,\infty)\), \(\ell \in \{1, \dotsc, m\}\) and let $u\colon \partial \R^m_+\to \n$ be a Borel--measurable map.
Assume further that \(w \colon \partial \R^{m}_+ \to [0, \infty]\) is a Borel--measurable function, such that for every finite homogeneous simplicial complex $\Sigma$ of dimension \(\ell\), any subcomplex $\Sigma_0 \subset \Sigma$ of dimension \(\ell - 1\), and every  Lipschitz--continuous map \(\sigma \colon \Sigma \to \partial \R^{m}_+\) satisfying 
\[
\int_{\Sigma_0} w \compose \sigma < \infty, 
\]
there exists a map \(W \in  \dot{W}^{1, p} (\Sigma, \n)\) with \(\tr_{\Sigma_0}W= u \compose \sigma \big\vert_{\Sigma_0}\) almost everywhere on \(\Sigma_0\) with the estimate 
\begin{equation}\label{eq:estimateofUonskeleton2}
  \int_{\Sigma} \abs{\deriv W}^p \\
  \le 
  \gamma^{\lambda}_{\Sigma_0, \Sigma}\, \seminorm{\sigma}_\lip^{p-1}
  \int_{\Sigma_0}  w \compose \sigma,
\end{equation}
where the quantity $\gamma^{\lambda}_{\Sigma_0, \Sigma} $ is defined in \eqref{eq:Alhforslike}.

Then, there exists a constant \(C>0\), such that for given \(h \in \partial \R^{m}_+\) and \(\kappa > 0\) for which
\[
 \smashoperator{\int_{\mathcal{C}^{\kappa, \ell - 1}_0 + h}}
 w < \infty,
\]
there exists a map \(V \in \dot{W}^{1, p} (\mathcal{C}^{\kappa, \ell}_+ + h, \n)\) satisfying \(\tr_{\mathcal{C}^{\kappa, \ell - 1}_0 + h} V = u\big\vert_{\mathcal{C}^{\kappa, \ell - 1}_0 + h}\) with the estimate 
 \[
   \smashoperator{\int_{\mathcal{C}^{\kappa, \ell}_+ + h}} \abs{\deriv V}^p \le C 
    \smashoperator{\int_{\mathcal{C}^{\kappa, \ell - 1}_0 + h}} w.
 \]
\end{proposition}

The cubications $\mathcal C_+^{\kappa, \ell}$ and $\mathcal{C}^{\kappa, \ell - 1}_0$ were defined in \eqref{eq:cubicationdefintion1}--\eqref{eq:cubicationdefintion3}.

It follows from \cref{pr:(2)to(3)} that \ref{it_aithesh6rea9ouGaib0Dae2f} implies \ref{it_ap6kahs3enainah2lohz9ET2} in \cref{theorem_main_halfspace}, since by Fubini's theorem
\[
{\int_{[0, \kappa]^{m - 1}}}
 {\int_{\mathcal{C}^{\kappa, \ell - 1}_0 + h}} w(x') \dif x' \dif h
 = \kappa^{m - \ell}
 \int_{\R^{m - 1}}w .
\]

\begin{proof}%
[Proof of \cref{pr:(2)to(3)}]%
\resetconstant
We define for \(j \in \N\) the sets 
\begin{align*}
  \Sigma^j &\defeq \{Q \in \mathcal{Q}^{1, \ell}_+ \colon Q \subset [-j, j]^m\}&
  &
  \text{and }&
 \Sigma^j_0 &\defeq \{Q \in \mathcal{Q}^{1, \ell - 1}_0 \colon Q \subset [-j, j]^m\},
\end{align*}
where $\mathcal{Q}^{1, \ell}_+$ and $\mathcal{Q}^{1, \ell - 1}_0$ are defined in \eqref{eq:cubicationdefintion1}--\eqref{eq:cubicationdefintion2}.

By a classical realization of cubes as simplicial complexes, we can assume that \(\Sigma^j\) is a simplicial complex of dimension $\ell$ and \(\Sigma^j_0\) is a simplicial subcomplex of \(\Sigma^j\) of codimension 1. Moreover, we observe that for every \(\lambda > 1\),
\[
 \Cl{cst_gohc2aich4bei1roochiiHah} \defeq \sup_{j \in \N} \gamma^\lambda_{\Sigma^j_0, \Sigma^j}
 < \infty.
\]
For \(\kappa > 0\) and \(h \in [0, \kappa]^{m - 1}\), we define the Lipschitz maps  \(\sigma^j_{\kappa, h} \colon \Sigma^j \to \partial\R^m_+ \simeq \R^{m-1}\) by 
\(\sigma^j_{\kappa, h} (y) \defeq \kappa y' + h\), where for $y\in \Sigma^j$ we write $y=(y',y_m)\subset \R^{m-1}\times \R$.
By assumption, we have 
\begin{equation}
\label{eq_uujo1Aizi4Liec5ahNgaekee}
  \int_{\Sigma^j_0} w \compose \sigma^j_{\kappa, h} 
  \le \frac{1}{\kappa^{\ell - 1}} \int_{\mathcal{C}^{\kappa, \ell - 1}_0 + h} w <\infty.
\end{equation}
By assumption, there exists a map \(W^j \in \dot{W}^{1, p} (\Sigma^j, \n)\) such that \(\tr_{\Sigma^j_0} W^j = u \compose \sigma^j_{\kappa, h} \big\vert_{\Sigma^j_0}\) and 
\begin{equation}
\label{eq_oocheiDai0bu8iujei7beeBiA}
 \int_{\Sigma^j} \abs{\deriv W^j}^p \le \Cr{cst_gohc2aich4bei1roochiiHah} \kappa^{p-1} \int_{\Sigma^j_0} w \compose \sigma^j_{\kappa, h}.
\end{equation}
In view of \eqref{eq_uujo1Aizi4Liec5ahNgaekee} and \eqref{eq_oocheiDai0bu8iujei7beeBiA}, we have
\begin{equation}
\label{eq_oocheiDai0bu8iujei7beeBi}
\begin{split}
 \int_{\sigma^{j}(\Sigma^j)} \abs{\deriv W^j \compose (\sigma^j_{\kappa, h})^{-1}}^p 
 &= \frac{1}{\kappa^{p - \ell}}
 \int_{\Sigma^j} \abs{\deriv W^j}^p \\
 & \le \Cr{cst_gohc2aich4bei1roochiiHah} \kappa^{\ell - 1}\int_{\Sigma^j_0} w \compose \sigma^j_{\kappa, h} \le \Cr{cst_gohc2aich4bei1roochiiHah}\int_{\mathcal{C}^{\kappa, \ell - 1}_0 + h} w.
\end{split}
\end{equation}

In view of \eqref{eq_oocheiDai0bu8iujei7beeBi} and by weak compactness in Sobolev spaces, up to a subsequence, we have \(W^j \compose (\sigma^j_{\kappa, h})^{-1}\to V\) almost everywhere on \(\mathcal{C}^{\kappa, \ell}_+ + h\), where \(V \in \dot{W}^{1, p} (\mathcal{C}^{\kappa, \ell}_+ + h, \n)\), \(\tr_{\mathcal{C}^{\kappa, \ell - 1}_0 + h} V= u\big\vert_{\mathcal{C}^{\kappa, \ell - 1}_0 + h}\) and 
\[
    \smashoperator{\int_{\mathcal{C}^{\kappa, \ell}_+ + h}} \abs{\deriv V}^p \le \C
    \smashoperator{\int_{\mathcal{C}^{\kappa, \ell - 1}_0 + h}} w.\qedhere
\]
\end{proof}

\subsection{From cubical meshes to the half-space}\label{ss:3implies1}
We now prove the implication \ref{it_ap6kahs3enainah2lohz9ET2} \(\implies\) \ref{it_eiquoongeiJ8quahb5yu7hoo} in \cref{theorem_main_halfspace}. 

\begin{proposition}
\label{pr:(3)to(1)}
Assume that  $2 \le m \in \N$ and \(1 < p < m\).
There exists a constant \(C>0\) such that 
for every Borel--measurable map $u\colon \R^{m-1}\times\{0\} \to \n$, every \(\theta > 0\), every sequence \((\kappa_i)_{i \in \N}\) in \((0, \infty)\) converging to \(0\), and every sequence of sets \(H_j \subset [0, \kappa_j]^{m - 1}\), if 
 \begin{equation}
 \label{eq_ceiY9eigiu3shux2quootai7}
  \liminf_{j \to \infty} \frac{\abs{H_j}}{\kappa_j^{m - 1}} > 0 
 \end{equation}
 and if for each \(h \in H_j \times \{0\}\), 
 there exists \(V \in \dot{W}^{1, p} (\mathcal{C}^{\kappa_j, \floor{p}}_+ + h, \n)\) such that \(\tr_{{\mathcal{C}^{\kappa_j, \floor{p - 1}}_0 + h}} V = u \big\rvert_{\mathcal{C}^{\kappa_j, \floor{p - 1}}_0 + h}\) and 
 \begin{equation}
 \label{eq_ieghieTeengo3QuaiJiePoo4}
 \kappa_j^{m - \floor{p}} 
   \smashoperator{\int_{\mathcal{C}^{\kappa_j, \floor{p}}_+ + h}} \abs{\deriv V}^p \le \theta,
 \end{equation}
then there exists a mapping \(U \in \dot{W}^{1, p} (\R^m_+, \n)\) such that $\tr_{\R^{m-1}\times\{0\}} U = u$
and 
\begin{equation}
\label{eq_Uashe2lie4shiN5deitheeVa}
 \int_{\R^m_+} \abs{\deriv U}^p \le C \theta,
\end{equation}
where the cubications $\mathcal C_+^{\kappa_j, \lfloor p \rfloor}$ and $\mathcal{C}^{\kappa_j, \floor{p - 1}}_0$ are defined in \eqref{eq:cubicationdefintion1}--\eqref{eq:cubicationdefintion3}.
\end{proposition}

\begin{proof}
\textsc{Step 1. Construction of \(U^j_{h}\) by homogeneous extension. }
For each \(j \in \N\) and every \(h \in H_j\), we define the map \(U^j_{h} \colon \R^m_+ \to \n\) by a homogeneous extension.
In order to define the extension we begin by introducing a retraction of \(\R^m_+ \setminus \mathcal{E}^{\kappa_j, m - \floor{p}}\) onto \(\mathcal{C}^{\kappa_j, \floor{p - 1}}_+\), where the dual skeleton \(\mathcal{E}^{\kappa, \ell}\) is defined for \(\ell \in \{0, \dotsc, m\}\) and \(\kappa > 0\) by
\begin{equation*}
 \mathcal{E}^{\kappa,  \ell}
 \defeq \bigcup \{Q \colon Q \text{ is an \(\ell\)--dimensional face of \([0, \kappa]^m + k \kappa\), for \(k \in \Z^m\)}\}.
\end{equation*}
For \(\ell \in \{1, \dotsc, m\}\), we define the mapping
\(
P_{\kappa, \ell} \colon \mathcal{C}^{\kappa, \ell}_+ \to \mathcal{C}^{\kappa, \ell - 1}_+
\) 
to be the homogeneous retraction defined on each cube $Q \in \mathcal{Q}^{\kappa, \ell}$ in the following way: Let $x_{Q}$ be the center of the cube $Q$, so that \(Q \cap \mathcal{E}^{\kappa, m - \ell} = \{x_Q\}\) (note that when $Q \cap \partial \R^{m}_+\neq \emptyset$ then $Q \cap \R^m_+$ is a half-cube and $x_Q \in \partial \R^m_+$). On this cube the map $P_{\kappa, \ell} \colon Q \setminus \{x_Q\} \to \partial Q$ (with the boundary taken in the \(\ell\)--dimensional affine plane containing \(Q\)) is given by the formula
\begin{equation}
\label{eq_vuib0ViengaQueineiy9xaxi}
 P_{\kappa, \ell} (x) \defeq x_{Q} + \kappa\frac{x-x_{Q}}{|x-x_{Q}|_{\infty}}.
\end{equation}
We define now \(\mathcal{P}^{\kappa, \ell} \colon \R^m_+ \setminus \mathcal{E}^{\kappa, m - \ell- 1} \to \mathcal{C}^{\kappa, \ell}_+\)  by 
\begin{equation}
\label{eq:definitonofP}
 \mathcal{P}^{\kappa, \ell} \defeq P_{\kappa, \ell + 1} \compose \dotsb \compose P_{\kappa, m}.
\end{equation}
(The map \(\mathcal{P}^{\kappa, \ell}\) is illustrated in Figure \ref{figure_P_kappa_j}.)
For any $h \in H_j$, we define \(U^{j, h}: \R^m \to \n\) for almost every \(x \in \R^m_+\) by 
\begin{equation}
\label{eq_EmithiBeeY0tahthephoboh4}
 U^{j, h}(x) \defeq 
 V^{j, h} (\mathcal{P}^{\kappa_j, \floor{p}} (x - h) + h),
\end{equation}
where \(V^{j, h} \in \dot{W}^{1, p} (\mathcal{C}^{\kappa_j, \floor{p}}_+ + h, \n)\) is a map given by assumptions, such that \(\tr_{\mathcal{C}^{\kappa_j, \floor{p - 1}}_0 + h} V^{j, h} = u \big\vert_{\mathcal{C}^{\kappa_j, \floor{p - 1}}_0 + h}\) and \eqref{eq_ieghieTeengo3QuaiJiePoo4} holds. 

\begin{figure}
\begin{center}
\begin{tikzpicture}
\draw[dashed] (-7,0)--(7,0);
\node[right] at (6, 0.2) {$\partial \R^m_+$};

 \draw[shift={(0,1)},line width=1pt] (-6,0)-- (6,0);
 \foreach \y/\xtext in {-0.5,1,2}
 \draw[shift={(0,2*\y+1)},line width=1pt] (-6,0)-- (6,0);
 \foreach \x/\xtext in {0,...,6}
 \draw[shift={(2*\x,0)},line width=1pt] (-6,0)-- (-6,5);

\draw [line width=0.5pt,<->] (-6.1,0) -- (-6.1,0.998);
\node[left, scale=0.8] at (-6.1,0.5) {$\kappa_j/2$};
\draw [line width=0.5pt,<->] (-6.1,1) -- (-6.1,2.998);
\node[left, scale=0.8] at (-6.1,2) {$\kappa_j$};
\draw [line width=0.5pt,<->] (-6.1,3) -- (-6.1,5);
\node[left, scale=0.8] at (-6.1,4) {$\kappa_j$};

\foreach \x/\xtext in {0,...,5}
\foreach \y/\xtext in {0,1}
\foreach \a/\xtext in {0,90,180,270}
\draw[shift={(2*\x,2*\y)},rotate around={\a:(-5,2)}] [->, line width=0.4pt,color=gr] (-5,2) -- (-6,2);
\foreach \x/\xtext in {0,...,5}
\foreach \y/\xtext in {0,1}
\foreach \a/\xtext in {0,90,180,270}
\draw[shift={(2*\x,2*\y)},rotate around={\a:(-5,2)}] [->, line width=0.4pt,color=gr] (-5,2) -- (-6,3);
\foreach \x/\xtext in {0,...,5}
\foreach \y/\xtext in {0,1}
\foreach \a/\xtext in {0,90,180,270}
\draw[shift={(2*\x,2*\y)},rotate around={\a:(-5,2)}] [->, line width=0.4pt,color=gr] (-5,2) -- (-6,2.5);
\foreach \x/\xtext in {0,...,5}
\foreach \y/\xtext in {0,1}
\foreach \a/\xtext in {0,90,180,270}
\draw[shift={(2*\x,2*\y)},rotate around={\a:(-5,2)}] [->, line width=0.4pt,color=gr] (-5,2) -- (-6,1.5);

\foreach \x/\xtext in {0,...,5}
\foreach \a/\xtext in {0,270,180}
\draw[shift={(2*\x,0)},rotate around={\a:(-5,0)}] [->, line width=0.4pt,color=gr] (-5,0) -- (-6,0);
\foreach \x/\xtext in {0,...,5}
\foreach \a/\xtext in {0,270}
\draw[shift={(2*\x,0)},rotate around={\a:(-5,0)}] [->, line width=0.4pt,color=gr] (-5,0) -- (-6,1);
\foreach \x/\xtext in {0,...,5}
\foreach \a/\xtext in {0,270}
\draw[shift={(2*\x,0)},rotate around={\a:(-5,0)}] [->, line width=0.4pt,color=gr] (-5,0) -- (-6,.5);
\foreach \x/\xtext in {0,...,5}
\foreach \a/\xtext in {180, 270}
\draw[shift={(2*\x,0)},rotate around={\a:(-5,0)}] [->, line width=0.4pt,color=gr] (-5,0) -- (-6,-.5);
\foreach \x/\xtext in {0,...,5}
\foreach \y/\xtext in {0,...,2}
 \draw[shift={(2*\x,2*\y)}] [fill=re] (-5,0) circle (1.5pt);

\end{tikzpicture}
\end{center}
\caption{The mapping $\mathcal{P}^{\kappa, \ell}$}
\label{figure_P_kappa_j}
\end{figure}
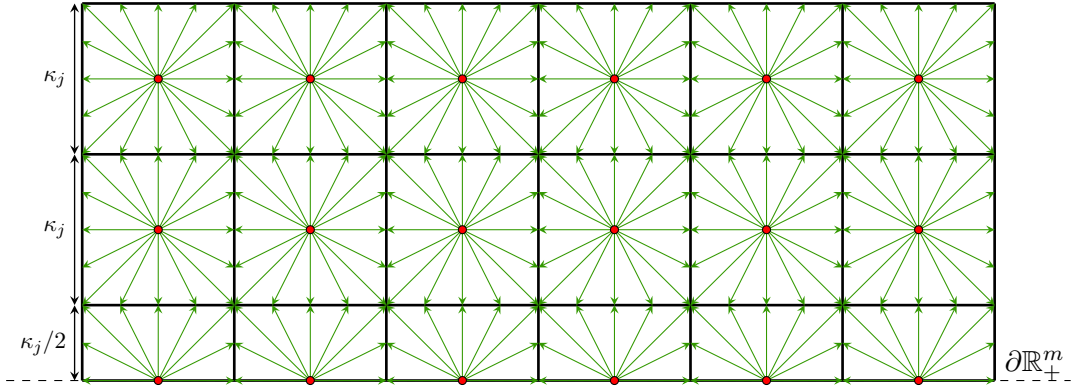

\textsc{Step 2. Uniform boundedness in $L^p$ of the gradients.}
We prove that when \(h \in H_j\), the sequence \((U^{j, h})_{j \in \N}\) remains bounded in \(\dot{W}^{1, p} (\R^m_+)\).
We begin with a well-known lemma.

\begin{lemma}\label{le:compositionwithprojection}
If \(\ell \in \N\), \(p < \ell\) and $V\in \dot{W}^{1,p}(\mathcal{C}^{\kappa, \ell-1}_+,\n)$, then we have $V\compose P_{\kappa,\ell} \in \dot{W}^{1,p}(\mathcal{C}^{\kappa, \ell}_+,\n)$ with the estimate
 \begin{equation}\label{eq:estimateoncompostionwithprojeciton}
  \int_{\mathcal{C}^{\kappa, \ell}_+ } |\deriv V \compose P_{\kappa, \ell}|^p \le \frac{2 \ell^{p/2}}{p - \ell} \kappa \int_{\mathcal{C}^{\kappa, \ell - 1}_+ } |\deriv V|^p .
 \end{equation}
\end{lemma}

Lemma \ref{le:compositionwithprojection} follows by the application of the next Lemma \ref{lemma_pyramid} on a suitable decomposition of cubes in pyramids (with a factor \(2\) coming from the fact that by definition of \(\dot{W}^{1,p}(\mathcal{C}^{\kappa, \ell-1}_+,\n)\) traces coincide on common faces).

\begin{lemma}
\label{lemma_pyramid}
Let \(\ell \in \N\) and for \(\kappa > 0\) we let 
\[
\Gamma_\kappa \defeq \{x=(x', x_\ell) \in [-\kappa,\kappa]^\ell \colon 0 \le x_\ell \le \kappa \text{ and } x_\ell = |x|_{\infty}\}.
\]
If  \(\ell > p\), \(f \in \dot{W}^{1, p} ([-\kappa, \kappa]^{\ell - 1}, \n)\), and if 
\(g \colon \Gamma_k \to \n\) is defined by 
\[
 g (x', x_\ell) = f (\kappa x'/x_\ell),
\]
then \(g \in \dot{W}^{1, p} (\Gamma_k, \n)\) and 
\[
  \int_{\Gamma_\kappa} \abs{\deriv g}^p 
  \le \frac{\ell^{p/2} \kappa }{p - \ell} 
  \int_{[-\kappa, \kappa]^{\ell-1}} \abs{\deriv f}^p.\ 
\]
\end{lemma}

\begin{proof}
We have 
\[
 \deriv g (x', x_\ell) = \kappa \Bigl(\frac{\deriv f (\kappa x'/x_\ell)}{x_\ell},
 - \frac{\deriv f (\kappa x'/x_\ell) \cdot x'}{x_\ell^2} \Bigr),
\]
so that 
\[
  \abs{\deriv g (x', x_\ell)}^2 
  = \kappa^2 \Bigl(\frac{\abs{\deriv f (\kappa x'/x_\ell)}^2}{x_\ell^2}
  + \frac{\abs{\deriv f (\kappa x'/x_\ell) \cdot x'}^2}{x_\ell^4}\Bigr)
  \le \ell \frac{\kappa^2}{x_\ell^2} \abs{\deriv f (\kappa x'/x_\ell)}^2,
\]
and thus, by Fubini's theorem and the change of variable \(y' = \kappa x'/x_\ell\), since \(p < \ell\), 
\[
\begin{split}
 \int_{\Gamma_\kappa} \abs{\deriv g}^p
 &\le  
 \ell^{p/2} 
 \int_0^\kappa 
  \int_{[-x_\ell, x_\ell]^{\ell - 1}}
  \kappa^{p} 
    \frac{\abs{\deriv f (\kappa x'/x_\ell)}^p}{x_\ell^p} \dif x' \dif x_\ell\\
&= \ell^{p/2} 
 \int_0^\kappa 
  \int_{[-\kappa, \kappa]^{\ell - 1}}
    \abs{\deriv f (y')}^p \frac{x_\ell^{\ell - p- 1}}{\kappa^{\ell - p - 1}} \dif y' \dif x_\ell
    = \kappa \frac{\ell^{p/2}}{\ell - p}
     \int_{[0, \kappa]^{\ell - 1}}
    \abs{\deriv f (y')}^p \dif y'.\qedhere
  \end{split} 
\]
\end{proof}

We pursue the proof of \cref{pr:(3)to(1)}.
Iterating Lemma \ref{le:compositionwithprojection}, in view of \eqref{eq_EmithiBeeY0tahthephoboh4},  \eqref{eq:definitonofP}, and \eqref{eq_ieghieTeengo3QuaiJiePoo4}, we obtain that for every \(h \in H_j\), we have \(U^{j, h} \in \dot{W}^{1, p} (\R^m_+, \n)\) with the estimate 
\begin{equation}
\label{eq_thic0eichaiB4ThaCie7BuHa}
\begin{split}
 \int_{\R^m_+} \abs{\deriv U^{j, h}}^p 
 = 
 \int_{\R^m_+} |\deriv V^{j, h} \compose(\mathcal{P}^{\kappa_j, \floor{p}} (\cdot -h) + h)|^p  
 &\le \C \kappa^{m-\floor{p}}\int_{\mathcal{C}^{\kappa_j,\floor{p}}_++h} \abs{\deriv V^{j, h}}^p  \leq \C \theta,
\end{split}
\end{equation}
the constants in the estimates depend only on $m$ and $p$.

\textsc{Step 3. Convergence of the boundary data.} 
We follow Brezis and Mironescu's proof, see \cite{Brezis-Mironescu-approximation}*{Lemma 4.1 (Step 1)}.

By the definition of the map $U^{j, h}$ in \eqref{eq_EmithiBeeY0tahthephoboh4} and since \(U^{j, h} \in \dot{W}^{1, p} (\R^m_+, \n)\) (see \eqref{eq_thic0eichaiB4ThaCie7BuHa}), for every \(h \in H_j\) we have 
\[ 
  \tr_{\R^{m-1}} U^{j, h} 
  = 
  \tr_{\R^{m-1}} V^{j, h} (\mathcal{P}^{\kappa_j, \floor{p}}(\cdot - h) + h) \eqqcolon u^{j, h}, 
\] 
thus,
\[
 u^{j, h} =  u \compose (\mathcal{P}^{\kappa_j, \floor{p}} ( \cdot -h) + h)\big\vert_{\R^{m - 1}\times \{0\}}.
\]
We are going to show that for a suitable choice of \(h_j \in  H_j\), \(u^{j, h_j} \to u\) in \(L^p_{\mathrm{loc}} (\R^{m - 1})\). 

We start with the observation that if \(k \in \Z^{m - 1} \times \{0\}\), then \(\mathcal{P}^{\kappa_j, \floor{p}} (x - \kappa k) = \mathcal{P}^{\kappa_j, \floor{p}} (x) - \kappa k\), and thus the map \(\partial \R^{m}_+\ni x  \mapsto x - \mathcal{P}^{\kappa_j, \floor{p}}(x) \in \partial \R^{m}_+\) is periodic.

\begin{lemma}\label{le:mainconvergenceofbdrydata} Let $f\colon \R^{\ell} \to \n$ be a Borel--measurable function and let \(\Psi \in L^\infty (\R^\ell)\). Assume that for every \(k \in \Z^\ell\), \(\Psi (x + \kappa k) = \Psi (x)\). Then, 
we have for every Borel--measurable set \(A \subset \R^\ell\)
\[ 
\int_{[0, \kappa]^\ell}  
\int_{A} \abs{f(x) - f(x - \Psi (x - h))}^p \dif x \dif h
\le \kappa^\ell \sup_{\abs{h} \le \norm{\Psi}_{L^\infty(\R^\ell)}} \int_{A} \abs{f  - f(\cdot -h)}^p.
\] 
\end{lemma} 

\Cref{le:mainconvergenceofbdrydata} is reminiscent to the opening of maps \cite{Brezis_Li}*{Section 1.1} and the related estimates \cite{BPVS2015}.

\begin{proof}
Since the function \(\Psi\) is periodic, we have by Fubini's theorem and the change of variable \(z = x - h\)
\[ 
\begin{split} 
\int_{[0, \kappa]^\ell} 
\int_{A}\abs{f(x) - f(x - \Psi (x - h))}^p \dif x \dif h
&= 
\int_{A}\int_{x - [0, \kappa]^\ell} 
 \abs{f(x) - f(x - \Psi (z))}^p  \dif z\dif x\\
&= \int_{A}\int_{[0, \kappa]^\ell} 
 \abs{f(x) - f(x - \Psi (z))}^p  \dif z\dif x\\
&\le \kappa^\ell \sup_{\abs{z} \le \norm{\Psi}_{L^\infty(\R^\ell)}} \int_{A} \abs{f  - f(\cdot -z)}^p.\ \qedhere
\end{split} 
\] 
\end{proof} 
Continuing, the proof of \cref{pr:(3)to(1)}, we set \(\Psi (x) \defeq  x - \mathcal{P}^{\kappa_j, \floor{p}} (x) \), so that for every \(x, h \in \partial \R^{m}_+\), 
\(\mathcal{P}^{\kappa_j, \floor{p}} (x  - h) + h = x - \Psi (x -h)\).
By \cref{le:mainconvergenceofbdrydata} with $\ell=m-1$, $f=u$, and \(A = B_R\) we obtain
\[ 
\frac{1}{\kappa_j^{m-1}}
\int_{[0, \kappa_j]^{m - 1}} \brac{\int_{B_R} |u - u(\mathcal{P}^{\kappa_j, \floor{p}} (\cdot -h) + h)|^p}  \dif h\le \sup_{|h| \le \sqrt{\ell - 1} \,\kappa_j} \|u - u(\cdot - h)\|^p_{L^p (B_R)} . 
\] 
Since \(u \in L^p_{\mathrm{loc}} (\R^{m - 1})\), there exists a sequence \((R_j)_{j \in \N}\) diverging to \(\infty\) such that 
\[
\lim_{j \to \infty}
 \frac{1}{\kappa_j^{m-1}}
\int_{[0, \kappa_j]^{m - 1}} \int_{B_{R_j}} |u(x') - u(\mathcal{P}^{\kappa_j, \floor{p}} (x' -h) + h)|^p  \dif x'\dif h = 0.
\]
By our assumption \eqref{eq_ceiY9eigiu3shux2quootai7}, for \(j \in \N\) large enough, we can choose an \(h_j \in H_j \simeq H_j \times \{0\}\) such that
\begin{equation}
\label{eq_ungaengaiV9oet8Oochulai4}
 \lim_{j \to \infty} \int_{B_{R_j}} |u - u(\mathcal{P}^{\kappa_j, \floor{p}} (\cdot -h_j) + h_j)|^p   = 0.
\end{equation}
We set 
\begin{equation}
\label{eq_Ahbahbe5lab8shai2aimie9a}
 U_j \defeq U^{j, h_j}.
\end{equation}
 
\textsc{Conclusion.}
By \eqref{eq_thic0eichaiB4ThaCie7BuHa}, the sequence \((U_j)_{j \in \N}\) that we defined in \eqref{eq_Ahbahbe5lab8shai2aimie9a} is bounded in \(\dot{W}^{1, p} (\R^{m}_+, \n)\). Therefore, up to a subsequence it converges weakly in \(\dot{W}^{1,p}(\R^m_+, \R^\nu)\) to a map $U\in \dot{W}^{1,p}(\R^m_+, \R^\nu)$. 
Since \(\n\) is compact, by Rellich--Kondrachov's compactness theorem we have strong convergence in $L^p(B_R)$ for every \(R > 0\), which implies, up to a subsequence, convergence almost everywhere; hence $U$ also takes values in the manifold \(\n\) and thus $U\in \dot{W}^{1,p}(\R^m_+,\n)$. 
Finally, on the boundary we have 
\[
 \tr_{\R^{m-1}} U_j \to \tr_{\R^{m - 1}} U \quad \text{in } L^p_{\mathrm{loc}} (\partial \R^m_+, \n) \text{ as } j\to \infty,
\]
and thus in view of \eqref{eq_ungaengaiV9oet8Oochulai4}, we conclude by continuity of the trace that 
\(\tr_{\R^{m-1}} U = u\). 
\end{proof}

\subsection{A qualitative necessary condition}
\label{section_qualitative}
Isobe's characterization of the obstruction to the extension of Sobolev mappings \cite{Isobe} consisted of an analytical obstruction \ref{it_zuugh7Thai5ThiephahtoP8y} and a topological obstruction \ref{it_dahGaiquu0xejai3Tah8aipa}.
On the other hand, the characterization of \cref{theorem_main_halfspace} is essentially quantitative. 
As a complement to the proof of \cref{theorem_main_halfspace} and in preparation of the proof of \cref{proof_global_decomposition}, we state and prove the next qualitative necessary condition for the extension.

\begin{proposition}
\label{pr:necessaryVMO}
Let \(2 \le m \in \N\), \(p\in (1,\infty)\).
If \(u = \tr_{\partial \R^{m}_+} U\) for some \(U \in \dot{W}^{1, p} (\R^m_+,\n)\), then there exists a Borel--measurable function \(w \colon \partial \R^{m}_+ \to [0, \infty]\) with
\(
 \int_{\partial \R^{m}_+} w <\infty
\)
such that if $\Sigma$ is a finite simplicial complex of dimension at most \(\floor{p}\), $\Sigma_0 \subset \Sigma$ is a subcomplex of codimension 1, the map \(\sigma \colon \Sigma \to \oline{\R}^{m}_+\) is Lipschitz--continuous, \(\sigma (\Sigma_0) \subset \partial \R^m_+\), and
\[
\int_{\Sigma_0} w \compose \sigma < \infty, 
\]
then \(u \compose \sigma \big\vert_{\Sigma_0}\) is homotopic in \(\mathrm{VMO} (\Sigma_0, \n)\) to \(V \big\vert_{\Sigma_0}\) for some \(V \in C (\Sigma, \n)\); if moreover \(\dim \Sigma < p\), then \(u \compose \sigma\big\vert_{\Sigma_0}\) is equal almost everywhere to \(V \big\vert_{\Sigma_0}\) for some \(V \in C (\Sigma, \n)\).
\end{proposition}

We recall following \cite{Brezis_Nirenberg_1} that a mapping $V\colon \Sigma \to \R^\nu$ belongs to the space \(\mathrm{VMO} (\Sigma, \R^\nu)\) whenever $V$ is Borel--measurable and 
\[
  \lim_{\rho\to 0} \sup_{a\in \Sigma} \fint\limits_{B^\Sigma_\rho(a)}\fint\limits_{B^\Sigma_\rho(a)} \abs{V(x) - V(y)} \dif x \dif y=0;
\]
the space \(\mathrm{VMO}(\Sigma, \R^\nu)\) endowed with bounded mean oscillation semi-norm
\begin{equation}
\label{eq_hiekiejer8aeseesh1hoo2Ac}
 \norm{V}_{\mathrm{BMO}} = \sup_{\substack{\rho > 0\\ a \in \Sigma}}\ \fint\limits_{B^\Sigma_\rho(a)}\fint\limits_{B^\Sigma_\rho(a)} \abs{V(x) - V(y)} \dif x \dif y < \infty
\end{equation}
and the distance of convergence in measure is complete. We finally set
\[
 \mathrm{VMO} (\Sigma, \n)
 \defeq 
 \{V \in \mathrm{VMO} (\Sigma, \R^\nu) \colon V (x) \in \n \text{ for almost every } x \in \Sigma\}.
\]
In particular, the maps \(V_0, V_1 \in \mathrm{VMO} (\Sigma, \n)\) are homotopic in \(\mathrm{VMO}(\Sigma, \n)\) whenever there exists a map \(H \in C ([0, 1], \mathrm{VMO}(\Sigma, \n))\) such that \(V_0 = H(0)\) and \(V_1 = H (1)\); the continuity is understood with respect to convergence in measure and convergence with respect to the bounded mean oscillation semi-norm \eqref{eq_hiekiejer8aeseesh1hoo2Ac}.

We recall that \cref{pr:trace_trails_necessary_2} gave under the same assumptions the conclusion that \(u \compose \sigma \big\vert_{\Sigma_0} = \tr_{\Sigma_0} W\) for some \(W \in \dot{W}^{1, p}(\Sigma, \n)\). 
\Cref{pr:necessaryVMO} would follow from \cref{pr:trace_trails_necessary_2}, embeddings of \(\dot{W}^{1, p} (\Sigma, \n)\) into \(\mathrm{VMO}(\Sigma, \n)\), an embedding of \(\dot{W}^{1-1/p, p} (\Sigma_0, \n)\) into \(\mathrm{VMO}(\Sigma_0, \n)\) together with a suitable approximation by continuous map. In order for this approach to work, one would need our assumption \(\dim \Sigma \le p\) together with a regularity assumption on the simplicial complex: for instance the embedding theorem fails for simplicial complex composed of two simplices intersecting on a set of codimension at least \(p\).

In order to avoid these technical issues, we follow \cite{BPVS} and give a direct proof of \cref{pr:necessaryVMO}.

\begin{proof}%
[Proof of \cref{pr:necessaryVMO}]%
\resetconstant
As in the proof of \cref{pr:trace_trails_necessary_2}, we 
define for \(y \in \partial \R^{m}_+\)
\begin{equation}
\label{eq_shu9adoeTh9sa1tevao2Jooz}
  w (y) \defeq 
  \left\{
  \begin{aligned}
  &\smashoperator[r]{\int_{\R^m_+}}
  \frac{W (x)x_m^\gamma}{\abs{x - y}^{m - 1 + \gamma}} \dif x&&\text{if
  } \lim_{\rho \to 0} \ \smashoperator{\fint_{\R^m_+ \cap B_\rho (y)}} \abs{u (y) - U} = 0,\\
  &\infty && \text{otherwise},\\
  \end{aligned}
  \right.
\end{equation}
where \(0 < \gamma \le 1\) will be chosen later. The function  \(W \colon \R^m_+ \to [0, \infty]\) is chosen as 
\begin{equation}
\label{eq_oongeikahhu1sheiS0Nai3ui}
 W (x) \defeq 
 \left\{
 \begin{aligned}
   &(\mathcal{M} \abs{\deriv U})^p(x) &&\text{ if
  } \lim_{\rho \to 0} \ \smashoperator{\fint_{\R^m_+ \cap B_\rho (x)}} \abs{U (x) - U} = 0,\\
   &\infty && \text{otherwise},
 \end{aligned}
 \right.
\end{equation}
with \(\mathcal{M} \abs{\deriv U}\colon \R^m_+ \to [0, \infty]\) denoting the Hardy--Littlewood maximal function, given for \(x \in \R^m_+\) by
\[
\mathcal{M} \abs{\deriv U}(x) 
 = \sup_{\delta > 0} \frac{1}{\abs{B_\delta (x)}} \smashoperator{\int_{B_\delta (x) \cap \R^m_+}} \abs{\deriv U}.
\]
By classical properties of Lebesgue points of functions and of traces (see for example \cite{Evans_Gariepy}*{Sections 1.7 and 5.3}), the first case in the definitions  \eqref{eq_shu9adoeTh9sa1tevao2Jooz} and \eqref{eq_oongeikahhu1sheiS0Nai3ui} is taken almost everywhere. By the classical Hardy--Littlewood maximal function theorem, since \(p > 1\), 
\begin{equation}
\begin{split}
\int_{\partial \R^{m}_+} w 
&= \int\limits_{\partial \R^{m}_+}
 \smashoperator[r]{\int_{\R^m_+}}
  \frac{W (x)x_m^\gamma }{\abs{x - y}^{m-1+\gamma}} \dif x \dif y = \Cl{cst_eequaY2kah0ew9Jie7logh0y} \int_{\R^{m}_+} W\\
 &= \Cr{cst_eequaY2kah0ew9Jie7logh0y} \int_{\R^{m}_+} (\mathcal{M} \abs{\deriv U})^p\le \C \int_{\R^{m}_+} \abs{\deriv U}^p < \infty.
 \end{split}
\end{equation}

Arguing as in the proof of \cref{pr:trace_trails_necessary_2}, since \(\gamma \le 1\), we obtain by an application of \cref{le:integration_2} that given an $\ell$--dimensional simplicial complex $\Sigma$, a subcomplex $\Sigma_0\subset \Sigma$ of codimension 1, a Lipschitz--continuous $\sigma\colon \Sigma\to \R^{m}_+$ such that $\sigma(\Sigma_0)\subset \partial \R^{m}_+$ and $\int_{\Sigma_0} w \compose \sigma <\infty$, there exists a \(\xi \in C^\eta_\rho\) such that 
\[
 \int_{\Sigma} W \compose \sigma_\xi < \infty,
\]
where the map $\sigma_\xi \colon \Sigma \to \R^m_+$ is defined in \eqref{eq_aiqu4quohn9aeSheeghoeshi} and the solid cone $C^\eta_\rho \subset \R^m$ is defined in \eqref{eq:solidsphericalcup}.

If \(W (x) < \infty\) then by definition of \(W\) in \eqref{eq_oongeikahhu1sheiS0Nai3ui}, \(x\) is Lebesgue point of \(U\) and for each \(\rho > 0\), by a suitable version of the Sobolev representation formula, we have in view of the definition of the maximal function
\begin{equation}
\label{eq_Ohr5aeghoomaeWee5oosae5e}
 \smashoperator{\fint_{\R^m_+ \cap B_\rho (x)}} \abs{U (x) -  U (z)} \dif z 
 \le \C \smashoperator{\int_{\R^m_+ \cap B_\rho (x) }} \frac{\abs{\deriv U (z)}}{\abs{z - x}^{m - 1}} \dif z
 \le \C \rho \mathcal{M} \abs{\deriv U} (x).
\end{equation}
If \(\abs{x - y} \le \rho\), then 
\(
  B_{\rho/2} (\tfrac{x + y}{2}) \subseteq B_\rho (x) \cap B_\rho (y)
\).
If \(x, y \in \Sigma\) and \(\abs{y - x} \le \delta\), then \(\abs{\sigma_{\xi} (x) - \sigma_{\xi} (y)} \le \seminorm{\sigma}_{\lip}\delta\). Thus, taking \( \rho \coloneqq  \seminorm{\sigma}_{\lip}\delta\), we obtain, in view of \eqref{eq_Ohr5aeghoomaeWee5oosae5e}, the Lusin--Lipschitz inequality
\begin{equation}
\label{eq_eiDaephioy5caquieNg3ouch}
\begin{split}
 \abs{U (\sigma_{\xi} (x))- U (\sigma_\xi(y))}
 &= \fint\limits_{B_{\frac{\rho}{2}}{\brk{(\sigma_\xi(x) + \sigma_\xi(y))/2}}} \abs{U (\sigma_{\xi} (x))- U (\sigma_\xi(y))} \dif z\\
 &\le \C \biggl(\smashoperator[r]{\fint_{B_\rho (\sigma_\xi (x))}} \abs{U (\sigma_\xi(x))- U (z)} \dif z + \smashoperator{\fint_{B_\rho (\sigma_\xi (y))}} \abs{U (\sigma_\xi(y)) -  U (z)} \dif z\biggr)\\
 &\le \C \rho \bigl(\mathcal{M} \abs{\deriv U} (\sigma_\xi (x)) + \mathcal{M} \abs{\deriv U} (\sigma_\xi(y))\bigr). 
 \end{split}
\end{equation}
Let $a\in \Sigma$ and \(\delta > 0\). Taking the mean value over $x,\,y\in \Sigma\cap B_\delta(a)$ on both sides of \eqref{eq_eiDaephioy5caquieNg3ouch} with \(\rho = \seminorm{\sigma}_{\mathrm{Lip}} \delta\) and applying H\"older's inequality, we get, since \(\dim \Sigma = \ell\),
\begin{multline}
\label{eq_ohghoo7seixo3sareiCh9que}
 \fint\limits_{B_\delta^\Sigma (a)} \fint\limits_{B_\delta^{\Sigma} (a)}
 \abs{U (\sigma_\xi (x))- U(\sigma_\xi (y))} \dif x\dif y \le \frac{\C \seminorm{\sigma}_{\mathrm{Lip}}}{\delta^{\ell - 1}}
 \smashoperator{\int_{B_\delta^\Sigma (a)}} \mathcal{M} \abs{\deriv U} (\sigma_\xi (y)) \dif y\\
 \le \Cl{cst_hee9AiCoo6osh3axeirew3no} \seminorm{\sigma}_{\mathrm{Lip}} \delta^{1 - \frac{\ell}{p}} \biggl(\int_{B_\delta^\Sigma (a)} (\mathcal{M} \abs{\deriv U})^p \compose \sigma_\xi \biggr)^\frac{1}{p}
 = \Cr{cst_hee9AiCoo6osh3axeirew3no} \seminorm{\sigma}_{\mathrm{Lip}} \delta^{1 - \frac{\ell}{p}} \biggl(\int_{B_\delta^\Sigma (a)} W \compose \sigma_\xi \biggr)^\frac{1}{p},
\end{multline}
where the constants depend on \(\Sigma\).

Similarly, if \(w (x)<\infty\), then 

\begin{equation}
\label{eq_aer1aelee4aifiebaeQu5voo}
 \smashoperator{\fint_{\R^m_+ \cap B_\rho (y)}} \abs{u (y) -  U (z)} \dif z 
 \le \C \smashoperator{\int_{\R^m_+ \cap B_\rho (y) }} \frac{\abs{\deriv U (z)}}{\abs{z - y}^{m - 1}} \dif z.
\end{equation}
Hence, similarly as before, for $x,\, y\in\Sigma_0$ we have
\begin{equation}
\label{eq_IechooCohsoxoru3cee1ohli}
 \abs{U (\sigma_{\xi} (y))- U (\sigma_\xi(x))} 
\le \C \biggl(\smashoperator[r]{\int_{ \R^m_+ \cap B_\rho (\sigma_\xi(x)) }} \frac{\abs{\deriv U (z)}}{\abs{z - \sigma_\xi(x)}^{m - 1}} \dif z 
+ \smashoperator[r]{\int_{\R^m_+ \cap B_\rho (\sigma_\xi(y))}} \frac{\abs{\deriv U (z)}}{\abs{z - \sigma_\xi(y)}^{m - 1}} \dif z 
 \biggr). 
\end{equation}
Hence, choosing \(\rho = \delta \seminorm{\sigma}_{\mathrm{Lip}}\), taking the mean value over $x,\, y\in B_\delta^{\Sigma_0}(a)$ on both sides, and applying H\"{o}lder's inequality we obtain
\begin{equation}
\label{eq_eitahNeze5ahT9pheiG4Ahth}
\begin{split}
 \fint\limits_{B_\delta^{\Sigma_0} (a)} \smashoperator[r]{\fint_{B_\delta^{\Sigma_0} (a)}}
 &\quad
 \abs{U (\sigma_{\xi}(x)) -  U(\sigma_{\xi}(y))} \dif x \dif y\\
 &\le \C
 \fint_{B_\delta^{\Sigma_0} (a)} 
 \int_{\R^m_+ \cap B_\rho (\sigma_\xi (x))} \frac{\abs{\deriv U (z)}}{\abs{z - \sigma_{\xi} (x)}^{m - 1}} \dif z \dif x \\
 &\le \C
 \biggl(
 \fint_{B_\delta^{\Sigma_0} (a)}
 \int_{\R^m_+} \frac{\abs{\deriv U (z)}^p z_m^{\gamma}}{\abs{z - \sigma_{\xi} (x)}^{m - 1 +\gamma}} \dif z \dif x\biggr)^\frac{1}{p}\\
 & \hspace{4em}
 \biggl( \fint\limits_{B_\delta^{\Sigma_0} (a)}
 \int\limits_{\R^m_+ \cap B_\rho (\sigma_\xi(x))} \frac{1}{z_m^{\gamma/(p - 1)}\abs{z - \sigma_\xi(x)}^{m - 1 - \gamma/(p - 1)}} \dif z \dif x\biggr)^{1 - \frac{1}{p}}\\
 &=  \C \, \delta^{1 - \frac{\ell}{p}} \seminorm{\sigma}_{\mathrm{Lip}}^{1 - 1/p}
 \biggl(
 \int_{B_\delta^{\Sigma_0} (a)} w \circ \sigma\biggr)^\frac{1}{p},
\end{split}
\end{equation}
provided \(\gamma < p - 1\), with constants depending on \(\Sigma_0\).

Finally, if \(x\in\Sigma_0\),  \(y\in \Sigma\) and \(\abs{x - y} \le \delta\) so that \(\abs{\sigma_{\xi} (x) - \sigma_{\xi}(y)} \le \rho = \delta \seminorm{\sigma}_{\mathrm{Lip}}\), we have by \eqref{eq_Ohr5aeghoomaeWee5oosae5e}, \eqref{eq_aer1aelee4aifiebaeQu5voo}, and the triangle inequality,
\begin{equation}
 \abs{U (\sigma_{\xi} (x))- U (\sigma_\xi(y))}
 \le \C \brk[\Bigg]{\,\smashoperator[r]{\int_{\R^m_+ \cap B_\rho (\sigma_{\xi} (x))}} \frac{\abs{\deriv U (z)}}{\abs{z - \sigma_\xi(x)}^{m - 1}} \dif z
 + \rho \mathcal{M} \abs{\deriv U} (\sigma_\xi(y))
}.
\end{equation}
Taking the average with respect to \(x \in B_\delta^{\Sigma_0} (a)\) and \(y \in B_\delta^\Sigma (a)\), we proceed as in \eqref{eq_ohghoo7seixo3sareiCh9que} and \eqref{eq_eitahNeze5ahT9pheiG4Ahth}, and we obtain
\begin{equation}
\label{eq_IeN5aghahjie1aiLoh9Eot2j}
 \fint\limits_{B_\delta^{\Sigma_0} (a)} \fint\limits_{B_\delta^\Sigma (a)}
 \abs{U (\sigma_\xi (x))- U(\sigma_\xi (y))} \dif y \dif x
 \le \C \delta^{1 - \frac{\ell}{p}} \brk[\Bigg]{\seminorm{\sigma}_{\mathrm{Lip}}^{p - 1} \smashoperator{\int\limits_{B_\delta^{\Sigma_0} (a)}} w \compose \sigma + \seminorm{\sigma}_{\mathrm{Lip}}^{p} \smashoperator{\int_{B_\delta^\Sigma (a)}} W \compose \sigma_\xi }^\frac{1}{p}.
\end{equation}

By Lebesgue's dominated convergence theorem, 
we have in view of \eqref{eq_ohghoo7seixo3sareiCh9que}, \eqref{eq_eitahNeze5ahT9pheiG4Ahth}, and \eqref{eq_IeN5aghahjie1aiLoh9Eot2j},
\begin{gather}
\label{eq_biesohthu2Oolo5niree1eib}
\lim_{\delta \to 0} \sup_{a \in \Sigma}
\frac{1}{\delta^{1 - \frac{\ell}{p}}} 
  \fint_{B_\delta^\Sigma (a)} \fint_{B_\delta^\Sigma (a)}
 \abs{U (\sigma_\xi (x)) - U(\sigma_\xi (y))} \dif x \dif y=0,\\
 \label{eq_co4ah3shieGhuphaifei7uch}
 \lim_{\delta \to 0} \sup_{a \in \Sigma_0}
 \frac{1}{\delta^{1 - \frac{\ell}{p}}}
 \fint_{B_\delta^{\Sigma_0} (a)} \fint_{B_\delta^{\Sigma_0} (a)}
 \abs{U (\sigma_{\xi}(x)) - U(\sigma_{\xi}(y))}\dif x \dif y=0,\\
\intertext{and}
\label{eq_ieGaeFeifah9yebaa1joisoo}
 \lim_{\delta \to 0} \sup_{a \in \Sigma_0}
 \frac{1}{\delta^{1 - \frac{\ell}{p}}} \fint_{B_\delta^{\Sigma_0} (a)} \fint_{B_\delta^\Sigma (a)}
 \abs{U (\sigma_\xi (x))- U(\sigma_\xi (y))} \dif x \dif y=0.
\end{gather}

We define
\begin{align*}
 v_\delta(x) &\defeq \fint_{B_\delta^{\Sigma_0} (x)} u \compose \sigma &
&\text{ and } &
 V_\delta(x) &\defeq \fint_{B_\delta^\Sigma (x)} U \compose \sigma_\xi.
\end{align*}
By \eqref{eq_biesohthu2Oolo5niree1eib}, we have \(\lim_{\delta \to 0} \dist (V_\delta, \n) = 0\), by \eqref{eq_co4ah3shieGhuphaifei7uch} \(\lim_{\delta \to 0} \dist (v_\delta, \n) = 0\), and by \eqref{eq_ieGaeFeifah9yebaa1joisoo} \(\lim_{\delta \to 0} \norm{V_\delta - v_\delta}_{L^\infty (\Sigma_0)} = 0\). Hence for \(\delta > 0\) small enough, \(\Pi_\n \compose V_\delta\) and \(\Pi_\n \compose v_\delta\) are well-defined and continuous respectively on \(\Sigma\) and on \(\Sigma_0\) and are close to each other on \(\Sigma_0\). Hence for \(\delta > 0\) small enough, \(\Pi_\n \compose v_\delta\) is homotopic to the restriction of a continuous map. Since \(\Pi_\n \compose v_\delta \to u \compose \sigma\big\vert_{\Sigma_0}\) in \(\mathrm{VMO} (\Sigma_0, \n)\), we conclude that \(u \compose \sigma \big\vert_{\Sigma_0}\) is homotopic in \(\mathrm{VMO}(\Sigma_0, \n)\) to the restriction to \(\Sigma_0\) of a continuous map on \(\delta\).
The case \(p > \ell\) follows from  \eqref{eq_aer1aelee4aifiebaeQu5voo}, \eqref{eq_eitahNeze5ahT9pheiG4Ahth}, \eqref{eq_IeN5aghahjie1aiLoh9Eot2j}, and Campanato's characterization of H\"older continuous functions by averages \cite{Campanato}. 
\end{proof}

\section{The global case}

In this Section we give the proof of \cref{theorem_global_necessary_sufficient} and \cref{proof_global_decomposition}. 

\subsection{Embedding into the half-space}
In order to reduce the situation of manifolds to an open subset of a Euclidean half-space, we rely on the following isometrical embedding.

\begin{proposition}[Isometrical embedding into a half-space and retraction]\label{proposition_mbdry_embedding_and_retraction} 
If \(\m\) is a compact Riemannian manifold with boundary \(\partial \m\), 
then there exists an isometric embedding \(i \colon \m \to \oline{\R}^\mu_+\) such that \(i (\partial \m) = i (\m) \cap \partial \R^\mu_+\),
and there exists a smooth map \(\Pi_{\m} \colon  \mathcal{U} \to i(\m)\) such that \(\mathcal{U} \subset  \oline{\R}^\mu_+\) is relatively open, 
\(\Pi_{\m} (\mathcal{U} \cap \partial \R^\mu_+) \subset i(\partial \m)\),
and for every \(x \in i (\m)\), \(x \in \mathcal{U}\), and \(\Pi_{\m} (x) = x\).
\end{proposition}

\begin{proof}
By a collar neighborhood theorem, we can assume that \(\m \subset \m'\), where \(\m'\) is a compact Riemannian manifold without boundary and the inclusion is an isometry.
We consider a function \(f \colon \m' \to \R\) such that \(f^{-1} (0) = \partial \m\), \(f^{-1} ([0, \infty)) = \m\) and \(0 <\abs{\deriv f} < 1\) on \(\m'\) with respect to the metric \(g'\) of \(\m'\).
In particular, \(g_0 \defeq g' - \deriv f \otimes \deriv f\) also defines a metric.
By Nash's embedding theorem, there exists a \(\mu \in \N\) and an  embedding \(i_0 \colon \m' \to \R^{\mu - 1}\) which is isometric for the metric \(g_0\).
The mapping \(i' \colon \m' \to \R^{\mu}\) defined by \(i' (x) = (i_0 (x), f (x))\) is then an isometric embedding for \(\m'\) endowed with the metric \(g'\) and \(i\coloneqq i' \big\vert_{\m} \colon \m \to \R^\mu\) is the required embedding.

Since \(i' (\m')\) is an embedded submanifold of the Euclidean space \(\R^\mu\), there exists an open set \(\mathcal{U}' \subset \R^\mu \) and \(\Pi' \in C^\infty (\mathcal{U}', i' (\m'))\) such that \(i'(\m') \subset \mathcal{U}'\) and for every \(x \in i'(\m')\), \(\Pi' (x) = x\).
Moreover, since \(\deriv f \ne 0\) in the construction of the embedding \(i'\), the submanifolds \(i' (\m')\) and \(\partial \R^\mu_+\) are transverse; 
there exists thus an open set \(\mathcal{U}^* \subset \R^{\mu}\) and a \(\delta > 0\) 
such that \(i' (\m') \cap (\R^{\mu - 1}\times (-\delta, \delta)) \subset \mathcal{U}^*\) and a mapping \(\Pi^* \in C^\infty (\mathcal{U}^*, i (\m))\) such that for every \((x', x_\mu) \in \mathcal{U}^*\), 
\(\Pi^* (x', x_\mu) \in \R^{\mu - 1} \times \{x_\mu\}\)
and for every \(x \in i (\m') \cap \mathcal{U}^*\), \(\Pi^* (x) = x\).
By taking the set \(\mathcal{U}^*\) smaller if necessary, we can also assume that for every \(x \in \mathcal{U}^*\) and every \(t \in [0, 1]\), we have \((1 -t) \Pi^* (x) + t x \in \mathcal{U}'\).
We conclude by defining the set
\[
 \mathcal{U} \defeq \brk[\big]{\mathcal{U}^* \cap \oline{\R}^\mu_+} \cup \brk[\big]{\mathcal{U}' \cap ( \R^{\mu - 1}\times (\delta/2, \infty))}, 
\]
and the mapping \(\Pi_{\m} \colon \mathcal{U} \to i (\m)\) for \(x = (x', x_\mu) \in \mathcal{U}\) by
\[
 \Pi_{\m} (x', x_\mu) \defeq \Pi' \brk[big]{\brk{1 - \psi (x_\mu)} \Pi^* (x) + \psi (x_\mu) x},
\]
where the function \(\psi \in C^\infty (\R, [0, 1])\) is taken to satisfy \(\psi (t) = 1\) when \(t \le \delta/2\) and \(\psi (t) = 0\) when \(t \ge \delta\).
\end{proof}

\subsection{Characterization of the trace space}\label{ss:mainthmonmanifold} 
We are now ready to prove \cref{theorem_global_necessary_sufficient} characterizing the traces of Sobolev maps between manifolds. The idea is to first use \cref{proposition_mbdry_embedding_and_retraction} to replace maps with a manifold in the domain to maps defined on a subset of the Euclidean half-space, by composing original maps with the retraction, and next to apply to those modified maps a localized version of \cref{theorem_main_halfspace}.

\begin{proof}%
[\hypertarget{proof:proof}{Proof of \cref{theorem_global_necessary_sufficient}}]%
\resetconstant%
Applying \cref{proposition_mbdry_embedding_and_retraction} we may assume, without loss of generality, that the manifold \(\m\) is identified with its isometrical embedding into the half-space \(\oline{\R}^\mu_+\) and that 
\(\Pi_{\m} \colon  \mathcal{U} \to \mathcal{M}\)
is the corresponding smooth retraction, where the set $\mathcal U \subset \R^{\mu-1}\times [0, \infty)$ is relatively open in the closed half-space \(\oline{\R}^\mu_+\).
We define the sets
\begin{align}\label{eq:Udefiniton}
 \mathcal{U}_0 &\coloneqq \mathcal{U} \cap \partial \R^\mu_+ &
 &\text{ and} &
 \mathcal{U}_+ &\coloneqq \mathcal{U} \cap \R^\mu_+;
\end{align}
we choose the set \(\mathcal{V} \subset \oline{\R}^\mu_+\) relatively open in \(\oline{\R}^\mu_+\) such that \(\mathcal{M} \subset \mathcal V\) and \(\oline{\mathcal{V}} \subset \mathcal{U}\), and  define the sets 
\begin{align}
\label{eq:Vdefinition}
 \mathcal{V}_0 &\defeq \mathcal{V} \cap \partial \R^\mu_+ &
 &\text{ and} &
 \mathcal{V}_+ &\defeq \mathcal{V} \cap \R^\mu_+.
\end{align}

\textsc{Necessary condition.} Fix $\delta>0$. By assumption there exists a map \(U \in \dot{W}^{1, p} (\m, \n)\) such that \(\tr_{\partial \m} U = u\). We define the maps \(\Bar{U} \defeq U \compose \Pi_{\m}\big\vert_{\mathcal{U}_+}\) and \(\Bar{u} \defeq u \compose \Pi_{\m}\big\vert_{\mathcal{U}_0}\), so that in particular \(\Bar{U} \in W^{1, p} (\mathcal{U}_+, \n)\) and \(\tr_{\mathcal{U}_0} \Bar{U} = \Bar{u}\).

We continue by observing that \cref{le:integration_2} and \cref{pr:trace_trails_necessary_2} with \(m = \mu \) admit localized versions. 
First, in \cref{le:integration_2} under the additional assumption that for each \(y \in \Sigma\) we have 
\begin{equation}
\label{eq_shiegeeyooB2oequei9jow9c}
\sigma(y) + d_0 (y) C^{\eta}_\rho \subseteq \mathcal{U}_+,
\end{equation}
the integral in \eqref{eq_aijalanguGh8oochoh2wuel6} can be taken over \(\mathcal{U}_+\) instead of \(\R^\mu_+\), and we thus have  
\begin{equation*}
  \int\limits_{C^{\eta}_{\rho}} \biggl(\int\limits_{\Sigma} F\compose (\sigma +  d_0\xi ) \biggr)\dif \xi
  \le \frac{(\eta \lambda - 1)^\mu(\rho + \seminorm{\sigma}_{\mathrm{Lip}})^{2 \mu}}{\eta^\mu\brk*{(\eta \lambda  - 1) \rho - \seminorm{\sigma}_{\mathrm{Lip}}}^{\mu + 1}}
  \gamma^{\lambda}_{\Sigma_0, \Sigma}
  \int\limits_{{\Sigma_0}} \int\limits_{\mathcal{U}_+} \frac{F (x) x_\mu}{\abs{x - \sigma (z)}^\mu}\dif x \dif z.
\end{equation*}
Indeed, it suffices to observe that the dimension has changed from \(m\) to \(\mu\) and that in view of  the condition \eqref{eq_shiegeeyooB2oequei9jow9c} and of the change of variable \(x = \sigma(y) +  d_0 (y)\xi\), the integration domain of all the integrals with respect to \(x\) can be restricted to the set \(\mathcal{U}_+\).

Next, for the localized version of \cref{pr:trace_trails_necessary_2}, we define the function $\Bar{W}\colon \mathcal{
U}_+\to [0,\infty]$ as in \eqref{eq_wiuw3je7aroo3eeG9vaez4ah}, with $\R^m_+$ replaced by $\mathcal{U}_+$ and $U$ by $\Bar{U}$, we define \(\Bar{w}\colon \mathcal{V}_0 \to [0,\infty]\) by \eqref{eq:defH}, with the integrals restricted to \(\mathcal{U}_+\); we have
\begin{equation}\label{eq:wbyU}
 \int_{\mathcal{V}_0} \Bar{w} \le C \int_{\mathcal{V}_+}|\deriv \Bar{U}|^p.
\end{equation}
By Fubini's theorem, there is an \(h \in \partial \R^\mu_+\) such that \(\partial \m + h \subset \mathcal{V}_0\) and \(\int_{\partial \m + h} \Bar{w} \le \C \int_{\mathcal{V}_0} \Bar{w} < \infty\).
Taking \(h\) to be small enough, we have that  \(\Pi_{\m} \big\vert_{\m + h}: \m + h \to \m\) is a diffeomorphism by the implicit function theorem.
We define \(w \defeq \Bar{w} \compose (\Pi_{\m} \big\vert_{\m + h})^{-1} \colon \partial \m \to [0, \infty]\).

If the mapping \(\sigma \colon  \Sigma \to \m\) is Lipschitz--continuous and if we set \(\Bar{\sigma}=  (\Pi_{\m} \big\vert_{\m + h})^{-1} \compose \sigma\), then
\(
 \seminorm{\Bar{\sigma}}_{\lip}
 \le \Cl{cst_eghietho0shee6eizeishiFe} \seminorm{\sigma}_{\lip}
\). Thus,  \(\seminorm{\sigma}_{\lip}\sup_{\Sigma} d_0\le \delta\) implies for \(\delta = \Bar{\delta}/\Cr{cst_eghietho0shee6eizeishiFe}\)
\begin{equation}\label{eq:smallnessofbarsigma}
\seminorm{\Bar{\sigma}}_{\lip}\sup_{\Sigma} d_0\le \Bar{\delta}. 
\end{equation}
Taking \(\eta = \frac{1}{2} + \frac{1}{2\lambda}\) in \eqref{eq:choiceofrho}, we get 
\begin{equation}\label{eq:currentrho}
\rho=4|\Bar{\sigma}|_{\lip}/(\lambda-1).
\end{equation}
Thus, for any $y\in\Sigma$ we have
\[
 \sigma(y) + d_0(y)C_{\rho}^\eta \subset \mathcal{V}_+ + \sup_{y\in\Sigma} d_0(y) B_\rho\cap \R^\mu_+ 
\]
and combining \eqref{eq:currentrho} with \eqref{eq:smallnessofbarsigma} we obtain
\[
 \sup_{y\in\Sigma} d_0(y)\rho \le \frac{4\Bar{\delta}}{\lambda-1} \le  \frac{4\Cr{cst_eghietho0shee6eizeishiFe} \delta}{\lambda-1}, 
\]
This implies, from the choice of the set $\mathcal{V}$, that for a sufficiently large $\lambda>1$ we have for all $y\in\Sigma$
\[
 \sigma(y) + d_0(y)C_{\rho}^\eta \subset \mathcal{U}_+
\]
and thus condition \eqref{eq_shiegeeyooB2oequei9jow9c} is satisfied. Moreover, since \(\Pi_{\m} \compose \Bar{\sigma} = \sigma\), we have
\[
 \int_{\Sigma_0} \Bar{w} \compose\Bar{\sigma}
 = 
 \int_{\Sigma_0} w \compose \sigma <\infty.
\]
We apply now localized \cref{le:integration_2} and proceed exactly as in the proof of \cref{pr:trace_trails_necessary_2}: For 
 the Lipschitz--continuous function \(\Bar{\sigma} \colon \Sigma \to \mathcal{V}_+\)  with \(\Bar{\sigma} (\Sigma_0) \subset \mathcal{V}_0\), \(\seminorm{\Bar{\sigma}}_{\lip} \le \Bar{\delta}\), and \(\int_{\Sigma_0} \Bar{w} \compose \Bar{\sigma}<\infty\) 
we obtain the existence of a map \(V \in \dot{W}^{1,p} (\Sigma, \n)\) such that \(\tr_{\Sigma_0} V = \Bar{u} \compose \Bar{\sigma} \big\vert_{\Sigma_0}=u \compose \sigma \big \vert_{\Sigma_0}\) and 
\begin{equation}
\label{eq_Fie1Nah8sah7reisirie3Rah}
 \int_{\Sigma} \abs{\deriv V}^p
 \le \gamma^{\lambda}_{\Sigma_0, \Sigma}\, \seminorm{\Bar{\sigma}}_\lip^{p - 1}
 \int_{\Sigma_0} \Bar{w} \compose \Bar{\sigma}\le \Cr{cst_eghietho0shee6eizeishiFe}^{p-1}\,\gamma^{\lambda}_{\Sigma_0, \Sigma} \, \seminorm{\sigma}_\lip^{p - 1}
 \int_{\Sigma_0} w \compose \sigma.
\end{equation}
Multiplying $w$ by a suitable constant we obtain \eqref{eq:againthesame}.
This finishes the proof of the necessity part.

\hypertarget{proof:sufficient}{\textsc{Sufficient condition:} } Let \(u \colon \partial \m \to \n\) and \(w \colon \partial \m \to [0, \infty]\) with 
\(
 \int_{\partial \m} w <\infty
\) be Borel--measurable maps given by assumptions. Since $\Pi_{\M}(\mathcal{U}_0)\subset \partial \m$, the map \(\Bar{w} \defeq w \compose \Pi_{\m} \colon \mathcal{U}_0 \to [0, \infty]\) is well-defined. If the mapping \(\Bar{\sigma} \colon  \Sigma \to \mathcal{U}\) is Lipschitz--continuous and if we set \(\sigma \defeq \Pi_{\m} \compose \Bar{\sigma}\colon \Sigma \to \m\), then 
\begin{equation*}
 \seminorm{\sigma}_{\lip} \le \Cl{cst_Lei6ahpee6imekaengohzae7} \seminorm{\Bar{\sigma}}_{\lip}
\end{equation*}
and 
\begin{equation*}
\int_{\Sigma_0} w \compose \sigma 
= \int_{\Sigma_0} \Bar{w} \compose \Bar{\sigma},
\end{equation*}
so that if $\Bar{\sigma}(\Sigma_0)\subset \mathcal{V}_0$, then for $\Bar{\delta} = \delta/\Cr{cst_Lei6ahpee6imekaengohzae7}$ the condition \(\seminorm{\Bar{\sigma}}_{\lip}\sup_\Sigma d_0 \le \Bar{\delta}\) implies $\seminorm{\sigma}_\lip \sup_{\Sigma} d_0 \le \delta$, and if \(\int_{\Sigma_0} \Bar{w} \compose \sigma  < \infty\), then by assumption there exists a map \(V \in \dot{W}^{1, p} (\Sigma, \n)\) such that \(\tr_{\Sigma_0} V = u \compose \sigma\big\rvert_{\Sigma_0}\) and 
\begin{equation}\label{eq:localVexistsinglobal}
 \int_{\Sigma} \abs{\deriv V}^p
 \le \gamma^{\lambda}_{\Sigma_0, \Sigma} \seminorm{\sigma}_\lip^{p - 1}
 \int_{\Sigma_0} w \compose \sigma.
\end{equation}
Thus by construction of \(\sigma\) and \(\Bar{w}\), \(\tr_{\Sigma_0} V = \Bar{u} \compose \Bar{\sigma}\big\rvert_{\Sigma_0}\), where we have set again $\Bar{u}\coloneqq u \compose \Pi_{\m}\big\rvert_{\mathcal{U}_0}$, and
\begin{equation*}
  \int_{\Sigma} \abs{\deriv V}^p
 \le \Cr{cst_Lei6ahpee6imekaengohzae7}^{p - 1} \gamma^{\lambda}_{\Sigma_0, \Sigma} \seminorm{\Bar{\sigma}}_\lip^{p - 1}
 \int_{\Sigma_0} \Bar{w} \compose \Bar{\sigma}.
\end{equation*}

For small enough \(\kappa_0 > 0\), we have by construction of \(\mathcal{U}_+\) and \(\mathcal{V}_+\)
\[
 \mathcal{V}_+ \subset \mathcal{W}_+
 \defeq \bigcup \{Q \cap \R^\mu_+ \in \mathcal{Q}^{\kappa_0, \mu} \colon Q \cap \R^\mu_+ \subset \mathcal{U}_+\}
\]
and 
\[
 \mathcal{V}_0 \subset \mathcal{W}_0 \defeq \bigcup\{Q\cap\partial\R^\mu_+\in \mathcal{Q}^{\kappa_0,\mu}\colon Q\cap\partial\R^\mu_+ \subset \mathcal{U}_0\},
\]
where the cubication $\mathcal{Q}^{\kappa_0,\mu}$ is defined as in \eqref{eq:cubicationdefintion1}, with $m$ replaced by $\mu$. 
We define for \(\ell \in \{1, \dotsc, \mu\}\) and \(j \in \N\), the sets
\begin{gather}
\label{eq_phieY8GiZohzae2yookepaif}
 \mathcal{W}^{j, \ell}_+ \defeq  \bigcup \{Q \cap 
\R^\mu_+ \colon Q \in \mathcal Q^{\kappa_j, \ell} \text{ and } Q \cap 
\R^\mu_+ \subset \mathcal{W}_+\},\\
 \label{eq_yich9dephaC0eeg0ooz5tieF}
\mathcal{W}^{j, \ell}_0 \defeq  \bigcup \{Q \cap \partial \R^\mu_+ \colon Q \in \mathcal Q^{\kappa_j, \ell + 1} \text{ and } Q \cap \partial \R^\mu_+ \subset \mathcal{W}_0\},\\
\intertext{and}
\mathcal{W}^{j, \ell} \defeq\mathcal{W}^{j, \ell}_+ \cup \mathcal{W}^{j, \ell-1}_0, 
\end{gather} 
where \(\kappa_j \defeq 2^{-j} \kappa_0\). 
If \(j\) is large enough,
then for every \(h \in [0, \kappa_j]^{\mu - 1} \times \{0\}\), we have
\[
 \m \subset \mathcal{W}^{j, \mu} + h \subset \mathcal{U}_+.
\]
We let $\Sigma^j$ be a sequence of homogeneous simplicial complexes, $\Sigma_0^j\subset \Sigma^j$ be a sequence of subcomplexes of codimension \(1\) 
and \(\sigma^j \colon \Sigma^j \to \mathcal{W}^{j, \floor{p}}\) be a simplicial parametrization such that \(\sigma^j (\Sigma^j_0) = \mathcal{W}_0^{j,\floor{p-1}}\). We observe that for every \(j \in \N\),
\begin{equation*}
 \abs{\sigma^j}_{\lip} \sup_{\Sigma^j} d_0 \le \Cl{constant},
\end{equation*}
so taking \(\Bar{\delta}=\Cr{constant}\) we get \(\abs{\sigma^j}_{\lip} \sup_{\Sigma^j} d_0 \le \Bar{\delta}\).
Moreover, we have for any $\lambda>0$
\begin{equation*}
\sup_{j \in \N} 
 \gamma^{\lambda}_{\Sigma^j_0, \Sigma^j} < \infty. 
\end{equation*}
Thus, as in \eqref{eq:localVexistsinglobal}, we obtain the existence of maps $V^j\in \dot{W}^{1,p}(\Sigma^j,\n)$. We then may proceed as in the proofs of \cref{pr:(2)to(3),pr:(3)to(1)} to construct a map \(\Bar{U} \in \dot{W}^{1, p} (\mathcal{V}_+, \n)\) such that \(\tr_{\mathcal{V}_0} \Bar{U} = \Bar{u}\) and
\[
 \int_{\mathcal{V}_+} \abs{\deriv \Bar{U}}^p
 \le \C \int_{\mathcal{U}_0} \Bar{w}
 \le \C \int_{\partial \m} w.
\]
By Fubini's theorem there is a set of positive measure of \(h \in \partial \R^\mu_+\), such that we have \(\m + h \subset \mathcal{V}\), \(\Bar{U}\big\vert_{\m + h} \in \dot{W}^{1, p} (\m + h, \n)\), \(\tr_{\partial \m + h} \Bar{U}\big\vert_{\m + h} = \Bar{u}\big\vert_{\partial \m + h}\), and 
\[
 \int_{\m + h } \abs{\deriv \Bar{U}}^p \le \C \int_{\mathcal{V}_+} \abs{\deriv \Bar{U}}^p.
\]
For such \(h\), we set \(U \defeq \Bar{U} \compose (\Pi_{\m} \big\vert_{\m + h})^{-1} \in \dot{W}^{1, p} (\m, \n)\) and we have \(\tr_{\partial \m} U = u\) on \(\partial \m\) and 
\begin{equation}\label{eq:Ubyw}
 \int_{\m} \abs{\deriv U}^p
 \le \C \int_{\partial \m} w.
\end{equation}
This finishes the proof of the sufficiency part for $\delta = \Cr{cst_Lei6ahpee6imekaengohzae7}\Cr{constant}$. 
\end{proof}

\begin{remark}
In view of \eqref{eq:wbyU} and \eqref{eq:Ubyw}, the infima of \(\int_{\M} \abs{\deriv U}^p\) and \(\int_{\partial \M} w\) are comparable.
\end{remark}

\subsection{Combining a qualitative and quantitative condition}\label{ss:combining}
In this Section we focus on proving \cref{proof_global_decomposition}. Let us first remark that if $p\notin \N$, then since $\dim \Sigma = \floor{p}$ and $V\in \dot{W}^{1,p}(\Sigma,\n)$ we obtain by the Morrey--Sobolev embedding and the homotopy extension property that the condition \ref{it_Chea0siechoa5uTieca2Xooc} is equivalent to the existence of \(V \in C (\Sigma, \n)\) such that \(V \vert_{\Sigma_0} = u \compose \sigma\) almost everywhere on \(\Sigma_0\).

The first tool of the proof is the following proposition about the extension of boundary data already in \(W^{1, p} (\partial \m, \n)\) which can be extended trivially to a neighborhood of the boundary.

\begin{proposition}
\label{proposition_trace_very_regular}
Let \(\M\) be a compact Riemannian manifold with boundary \(\partial \M\) and let \(u \in \dot{W}^{1,p} (\partial \M ,\n)\) be a Borel--measurable map.\ 

\noindent
Suppose that there exists a summable function \(w \colon  \partial \M \to [0,\infty]\) with the following property: if \(\Sigma\) is a homogeneous simplicial complex of dimension \(\floor{p}\), \(\Sigma_0\subset \Sigma\) is a subcomplex  of \(\Sigma\) of dimension \(\floor{p - 1}\), \(\sigma \colon  \Sigma \to \M\) is 
a Lipschitz--continuous map such that \(\sigma (\Sigma_0) \subseteq \partial \M\) satisfying \(\int_{\Sigma_0} w \compose \sigma < \infty\),
then \(u \compose \sigma \big\vert_{\Sigma_0}\) is homotopic in \(\mathrm{VMO} (\Sigma_0, \n)\) to the restriction of \(V \big\vert_{\Sigma_0}\) for some \(V \in C (\Sigma, \n)\).

\noindent
Then there exists an extension \(U \in \dot{W}^{1, p} (\m, \n)\) with \(\tr_{\partial \m} U = u\) on \(\partial \m\).
\end{proposition}
\begin{proof}
We use the same definitions as in the 
\hyperlink{proof:proof}{Proof of \cref{theorem_global_necessary_sufficient}}.
In particular, for the summable function $w\colon \partial \m \to [0,\infty]$ from the assumptions, we define $\Bar{w}\coloneqq w\compose \Pi_\m$.
If \(\Bar{\sigma} \colon  \Sigma \to \m\) is Lipschitz--continuous and if \(\int_{\Sigma_0} \Bar{w} \compose \Bar{\sigma}  < \infty\), then, defining \(\sigma \defeq \Pi_{\m} \compose \Bar{\sigma}\), we also have that \(\sigma\) is Lipschitz--continuous and \(\int_{\Sigma_0} w \compose \sigma < \infty\). 
Thus, by assumption,  the map \(\Bar{u} \compose \Bar{\sigma} = u \compose \sigma\) is homotopic in \(\mathrm{VMO} (\Sigma_0, \n)\) to the restriction of \(V \big\vert_{\Sigma_0}\) for some \(V \in C (\Sigma, \n)\), where $\Bar{u}=u\compose\Pi_{\m}\big\rvert_{\mathcal{U}_0}$.

Again, proceeding as in the proof of the \hyperlink{proof:sufficient}{\textsc{sufficient condition}} in \cref{theorem_global_necessary_sufficient}, we obtain that for some fixed large enough \(j \in \N\) and each $h\in[0,\kappa_j]^{\mu-1}\simeq [0,\kappa_j]^{\mu-1}\times\{0\}$, where $\kappa_j=2^{-j}\kappa_0>0$, we have
\[
 \m \subset \mathcal{W}^{j,\mu} +h \subset \mathcal{U}
\]
and for almost every \(h \in  [0, \kappa_j]^{\mu - 1} \simeq [0, \kappa_j]^{\mu - 1} \times \{0\}\), and for each \(\ell \in \{0, \dotsc, \mu - 1\}\), we have
\begin{equation}\label{eq:goodhforbaru}
  \Bar{u}\big\vert_{\mathcal{W}^{j, \ell}_0 + h} \in \dot{W}^{1, p} (\mathcal{W}^{j, \ell}_0 + h, \n)
\end{equation}
and  \(\Bar{u} \big\vert_{\mathcal{W}^{j, \floor{p - 1}}_0 + h}\)
is homotopic in \(\mathrm{VMO}(\mathcal{W}^{j, \floor{p - 1}}_0 + h, \n)\) to the restriction $V\big\rvert_{\mathcal{W}^{j, \floor{p - 1}}_0 + h}$ of a continuous map \(V \in C (\mathcal{W}^{j, \floor{p}} + h, \n)\). By a regularization argument and the homotopy extension property, there exists a \(\Bar{U}^{\floor{p}} \in C (\mathcal{W}^{j, \floor{p}} + h, \n)\) such that \(\tr_{\mathcal{W}^{j, \floor{p - 1}}_0 + h} \Bar{U}^{\floor{p}} = \Bar{u}\).

Now we define inductively maps $\Bar{U}^{\ell}$ for $\ell \in \{\floor{p+1},\dots,\mu\}$. Given \(\Bar{U}^{\ell - 1} \in \dot{W}^{1, p} (\mathcal{W}^{j, \ell-1}_+ + h, \n)\), we define \(\Bar{U}^\ell \in \dot{W}^{1, p} (\mathcal{W}^{j, \ell}_+ + h, \n)\) on each $Q\in \mathcal{Q}^{\kappa_j, \ell}$ by
\[
\Bar{U}^\ell \defeq\left\{
 \begin{array}{ll}
  \Bar{u} &\text{ if } Q \subset \mathcal{W}^{j, \ell}_0\\
  \Bar{U}^{\ell - 1} (P_{\kappa_j, \ell} (\cdot - h) + h) &\text{ if } Q \subset \mathcal{W}^{j, \ell}_+, \text{(but } Q \not \subset \mathcal{W}^{j, \ell}_0),
 \end{array}
 \right.
\]
where the projection $P_{\kappa_j, \ell}$ is defined in \eqref{eq_vuib0ViengaQueineiy9xaxi} and we take any $h\in[0,\kappa_j]^{\mu-1}$ on which \eqref{eq:goodhforbaru} holds. In view of \cref{lemma_pyramid}, we have \(\Bar{U}^\ell \in \dot{W}^{1, p} (\mathcal{W}^{j, \ell}_+ + h, \n)\).
We set \(\Bar{U} = \Bar{U}^\mu\) and in order to obtain \(U \in \dot{W}^{1, p} (\M, \n)\) we use a Fubini argument exactly as at the end of the proof of the \hyperlink{proof:sufficient}{\textsc{sufficient condition}} of \cref{theorem_global_necessary_sufficient}.
\end{proof}

The second tool is a localized version of \cref{pr:necessaryVMO}.

\begin{proposition}
\label{pr:necessaryVMO_manifold}
 Let \(2 \le m \in \N\) and \(p\in (1,\infty)\).
If \(u = \tr_{\partial \m} U\) for some \(U \in \dot{W}^{1, p} (\m, \n)\), then there exists a Borel--measurable function \(w \colon \partial \m \to [0, \infty]\) with
\(
 \int_{\partial \m} w <\infty
\)
such that if $\Sigma$ is a finite simplicial complex of dimension at most \(\floor{p}\), $\Sigma_0 \subset \Sigma$ is a subcomplex of codimension 1, \(\sigma \colon \Sigma \to \m\) is a Lipschitz--continuous  with $\sigma(\Sigma_0) \subset \partial \m$, and
\(
\int_{\Sigma_0} w \compose \sigma < \infty, 
\)
then \(u \compose \sigma\big\vert_{\Sigma_0}\) is homotopic in \(\mathrm{VMO} (\Sigma_0, \n)\) to \(V \big\vert_{\Sigma_0}\) for some \(V \in C (\Sigma, \n)\).
\end{proposition}

We omit the proof of \cref{pr:necessaryVMO_manifold} which is a straightforward variant of the proof of \cref{pr:necessaryVMO_manifold}.

\begin{proof}%
[Proof of \cref{proof_global_decomposition}]
\textsc{Necessary condition.}
The condition \ref{it_kuixoongeeleag7aeVeiyie4} follows immediately from \cref{theorem_global_necessary_sufficient} since the condition  \(\sigma (\Sigma) \subseteq \partial \m\) implies \(\sigma (\Sigma) \subseteq \m\) and \(\sigma(\Sigma_0) \subseteq \partial \m\); the condition \ref{it_Chea0siechoa5uTieca2Xooc} follows from \cref{pr:necessaryVMO_manifold}. 

\textsc{Sufficient condition.}
In view of the assumption in \ref{it_kuixoongeeleag7aeVeiyie4}, by a variant  of \cref{theorem_global_necessary_sufficient} applied to the manifold \(\partial \m \times (0, 1)\) with boundary \(\partial \m \times \{0\}\)\footnote{\Cref{theorem_global_necessary_sufficient}, is proved for \(\m\) compact; here we have \(\m = \partial\m \times [0, 1)\) to which the construction of the isometric embedding and retraction in \(\R^\mu_+\) also apply.}
 there exists a map
\(V \in \dot{W}^{1, p} (\partial \m \times [0, 1], \n)\) such that \(\tr_{\partial \m \times \{0\}} V = u\). By Fubini's theorem, for almost every \(t \in (0, 1]\), \(V \vert_{\partial\m \times \{t\}} = \tr_{\partial\m \times \{t\}} V \in \dot{W}^{1, p}(\partial\m \times \{t\}, \n)\). Through a suitable rescaling in the \(t\) variable we can assume without loss of generality that this is the case for \(t = 1\).

By \cref{pr:necessaryVMO_manifold} applied to the manifold \(\partial \m \times [0,1]\) with boundary \(\partial \m \times \{0, 1\}\) and the boundary map $V \vert_{\partial \m \times \{0,1\}}$, there exists a summable function \(\Tilde{w} \colon \partial \m \times \{0, 1\}\to[0,\infty]\) such that for every finite homogeneous simplicial complex \(\Tilde{\Sigma}\) of dimension \(\floor{p}\), every subcomplex \(\Tilde{\Sigma}_0\subset \Tilde{\Sigma}\) of codimension \(1\), every Lipschitz--continuous map \(\Tilde{\sigma} \colon \Tilde{\Sigma} \to \partial \M \times [0, 1]\) satisfying \(\Tilde{\sigma} (\Tilde{\Sigma}_0) \subset \partial \m \times \{0, 1\}\) and 
\(\int_{\Tilde{\Sigma}_0} \Tilde{w} \compose \Tilde{\sigma} <\infty\), the map \(V \compose \Tilde{\sigma} \big\vert_{\Tilde{\Sigma}_0}\) is homotopic in $\mathrm{VMO}(\Tilde{\Sigma}_0,\n)$ to the restriction to $\Tilde{\Sigma}_0$ of a continuous map from \(\Tilde{\Sigma}\) to \(\n\). 

In order to apply \cref{proposition_trace_very_regular}, we define $\Hat{w}\colon \partial \m \to [0,\infty]$ by \(\Hat{w} \defeq  \Tilde{w} (\cdot, 0) + \Tilde{w} (\cdot, 1) + w\), where $w$ is the summable map given by assumptions. If \(\Sigma\) is a finite \(\floor{p}\)--dimensional simplicial complex and \(\Sigma_0\) is a \(\floor{p - 1}\)--dimensional subcomplex, and if \(\sigma \colon \Sigma \to \m\) is Lipschitz--continuous such that \(\sigma (\Sigma_0) \subset \partial \m\), we define \(\Tilde{\Sigma}\) to be a simplicial realization of \(\Sigma_0 \times [0, 1]\) and  set \(\Tilde{\sigma} (y, t) \defeq (\sigma (y), t) \in \partial \m \times [0, 1]\). If \(\int_{\Sigma_0} \Hat{w} \compose \sigma < \infty\), then 
\(\int_{\Sigma_0 \times \{0, 1\}} \Tilde{w} \circ \Tilde \sigma< \infty\), and thus since \(\dim (\Sigma) = \floor{p} \le p\), the maps \(V\circ \Tilde{\sigma}\big\vert_{\Sigma_0 \times \{0\}}\) and \(V\circ \Tilde{\sigma}\big\vert_{\Sigma_0 \times \{1\}}\) are homotopic in \(\mathrm{VMO} (\Sigma_0, \n)\).
Moreover, in view of \ref{it_Chea0siechoa5uTieca2Xooc}, since \(\int_{\Sigma_0} w \circ \sigma< \infty\), the map \(u\circ\sigma\big \vert_{\Sigma_0} = V\big\vert_{\Sigma_0 \times \{0\}}\) is homotopic in \(\mathrm{VMO} (\Sigma_0, \n)\) to the restriction to $\Sigma_0$ of a map in \(C (\Sigma, \n)\). 
By transitivity of homotopies, \(V\big\vert_{\Sigma_0 \times \{1\}}\) is homotopic in \(\mathrm{VMO} (\Sigma_0, \n)\) to the restriction of a map in \(C (\Sigma, \n)\). By \cref{proposition_trace_very_regular}, \(V\big\vert_{\partial \m \times \{1\}} \in \dot{W}^{1, p} (\partial \m \times \{1\},\n)\) is the trace of a map in \(\dot{W}^{1, p} (\m, \n)\) and the conclusion follows.
\end{proof}

\bibliography{bib}%

\begin{bibdiv}
\begin{biblist}

\bib{Bethuel-obstruction}{article}{
      author={Bethuel, Fabrice},
       title={A new obstruction to the extension problem for {S}obolev maps
  between manifolds},
        date={2014},
        ISSN={1661-7738},
     journal={J. Fixed Point Theory Appl.},
      volume={15},
      number={1},
       pages={155\ndash 183},
         url={https://doi.org/10.1007/s11784-014-0185-0},
}

\bib{Bethuel_Chiron}{incollection}{
      author={Bethuel, Fabrice},
      author={Chiron, David},
       title={Some questions related to the lifting problem in {S}obolev
  spaces},
        date={2007},
   booktitle={Perspectives in nonlinear partial differential equations},
      series={Contemp. Math.},
      volume={446},
   publisher={Amer. Math. Soc., Providence, RI},
       pages={125\ndash 152},
         url={https://doi.org/10.1090/conm/446/08628},
}

\bib{Bethuel-Demengel}{article}{
      author={Bethuel, Fabrice},
      author={Demengel, Fran{\c c}coise},
       title={Extensions for {S}obolev mappings between manifolds},
        date={1995},
        ISSN={0944-2669},
     journal={Calc. Var. Partial Differential Equations},
      volume={3},
      number={4},
       pages={475\ndash 491},
         url={https://doi.org/10.1007/BF01187897},
}

\bib{BPVS}{article}{
      author={Bousquet, Pierre},
      author={Ponce, Augusto},
      author={Van~Schaftingen, Jean},
       title={Generic topological screening and approximation of sobolev maps},
        note={in preparation},
}

\bib{BPVS2015}{article}{
      author={Bousquet, Pierre},
      author={Ponce, Augusto~C.},
      author={Van~Schaftingen, Jean},
       title={Strong density for higher order {S}obolev spaces into compact
  manifolds},
        date={2015},
        ISSN={1435-9855},
     journal={J. Eur. Math. Soc. (JEMS)},
      volume={17},
      number={4},
       pages={763\ndash 817},
}

\bib{Brezis_Li}{article}{
      author={Brezis, Haim},
      author={Li, Yanyan},
       title={Topology and {S}obolev spaces},
        date={2001},
        ISSN={0022-1236},
     journal={J. Funct. Anal.},
      volume={183},
      number={2},
       pages={321\ndash 369},
}

\bib{Brezis-Mironescu-approximation}{article}{
      author={Brezis, Ha\"{\i}m},
      author={Mironescu, Petru},
       title={Density in {$W^{s,p}(\Omega;N)$}},
        date={2015},
        ISSN={0022-1236},
     journal={J. Funct. Anal.},
      volume={269},
      number={7},
       pages={2045\ndash 2109},
         url={https://doi.org/10.1016/j.jfa.2015.04.005},
}

\bib{Brezis_Nirenberg_1}{article}{
      author={Brezis, Ha\"{\i}m},
      author={Nirenberg, Louis},
       title={Degree theory and {BMO}. {I}. {C}ompact manifolds without
  boundaries},
        date={1995},
        ISSN={1022-1824},
     journal={Selecta Math. (N.S.)},
      volume={1},
      number={2},
       pages={197\ndash 263},
         url={https://doi.org/10.1007/BF01671566},
}

\bib{Brezis_Nirenberg_2}{article}{
      author={Brezis, Ha\"{\i}m},
      author={Nirenberg, Louis},
       title={Degree theory and {BMO}. {II}. {C}ompact manifolds with
  boundaries},
        date={1996},
        ISSN={1022-1824},
     journal={Selecta Math. (N.S.)},
      volume={2},
      number={3},
       pages={309\ndash 368},
         url={https://doi.org/10.1007/BF01587948},
}

\bib{Campanato}{article}{
      author={Campanato, S.},
       title={Propriet\`a di h\"{o}lderianit\`a di alcune classi di funzioni},
        date={1963},
        ISSN={0391-173X},
     journal={Ann. Scuola Norm. Sup. Pisa Cl. Sci. (3)},
      volume={17},
       pages={175\ndash 188},
}

\bib{Evans_Gariepy}{book}{
      author={Evans, Lawrence~C.},
      author={Gariepy, Ronald~F.},
       title={Measure theory and fine properties of functions},
      series={Studies in Advanced Mathematics},
   publisher={CRC Press, Boca Raton, FL},
        date={1992},
        ISBN={0-8493-7157-0},
}

\bib{Gagliardo}{article}{
      author={Gagliardo, Emilio},
       title={Caratterizzazioni delle tracce sulla frontiera relative ad alcune
  classi di funzioni in {$n$} variabili},
        date={1957},
        ISSN={0041-8994},
     journal={Rend. Sem. Mat. Univ. Padova},
      volume={27},
       pages={284\ndash 305},
         url={http://www.numdam.org/item?id=RSMUP_1957__27__284_0},
}

\bib{Hang_Lin_II}{article}{
      author={Hang, Fengbo},
      author={Lin, Fanghua},
       title={Topology of {S}obolev mappings. {II}},
        date={2003},
        ISSN={0001-5962},
     journal={Acta Math.},
      volume={191},
      number={1},
       pages={55\ndash 107},
         url={https://doi.org/10.1007/BF02392696},
}

\bib{HL1987}{article}{
      author={Hardt, Robert},
      author={Lin, Fang-Hua},
       title={Mappings minimizing the {$L^p$} norm of the gradient},
        date={1987},
        ISSN={0010-3640},
     journal={Comm. Pure Appl. Math.},
      volume={40},
      number={5},
       pages={555\ndash 588},
         url={https://doi.org/10.1002/cpa.3160400503},
}

\bib{Isobe}{article}{
      author={Isobe, Takeshi},
       title={Obstructions to the extension problem of {S}obolev mappings},
        date={2003},
        ISSN={1230-3429},
     journal={Topol. Methods Nonlinear Anal.},
      volume={21},
      number={2},
       pages={345\ndash 368},
         url={https://doi.org/10.12775/TMNA.2003.021},
}

\bib{Maly2017}{article}{
      author={Mal\'y, Luk\'a\v~s},
       title={Trace and extension theorems for {S}obolev-type functions in
  metric spaces},
      eprint={arXiv:1704.06344},
}

\bib{Maly_Shanmugalingam_Snipes_2018}{article}{
      author={Mal\'{y}, Luk\'{a}\v{s}},
      author={Shanmugalingam, Nageswari},
      author={Snipes, Marie},
       title={Trace and extension theorems for functions of bounded variation},
        date={2018},
        ISSN={0391-173X},
     journal={Ann. Sc. Norm. Super. Pisa Cl. Sci. (5)},
      volume={18},
      number={1},
       pages={313\ndash 341},
}

\bib{Mironescu_VanSchaftingen_Lifting}{article}{
      author={Mironescu, Petru},
      author={Van~Schaftingen, Jean},
       title={Lifting in compact covering spaces for fractional {S}obolev
  mappings},
        date={2021},
        ISSN={2157-5045},
     journal={Anal. PDE},
      volume={14},
      number={6},
       pages={1851\ndash 1871},
         url={https://doi.org/10.2140/apde.2021.14.1851},
}

\bib{Mironescu_VanSchaftingen_trace}{article}{
      author={Mironescu, Petru},
      author={Van~Schaftingen, Jean},
       title={Trace theory for {S}obolev mappings into a manifold},
        date={2021},
        ISSN={0240-2963},
     journal={Ann. Fac. Sci. Toulouse Math. (6)},
      volume={30},
      number={2},
       pages={281\ndash 299},
         url={https://doi.org/10.5802/afst.1675},
}

\bib{Monteil_VanSchaftingen}{article}{
      author={Monteil, Antonin},
      author={Van~Schaftingen, Jean},
       title={Uniform boundedness principles for {S}obolev maps into
  manifolds},
        date={2019},
        ISSN={0294-1449},
     journal={Ann. Inst. H. Poincar\'{e} Anal. Non Lin\'{e}aire},
      volume={36},
      number={2},
       pages={417\ndash 449},
         url={https://doi.org/10.1016/j.anihpc.2018.06.002},
}

\bib{Nash}{article}{
      author={Nash, John},
       title={The imbedding problem for {R}iemannian manifolds},
        date={1956},
        ISSN={0003-486X},
     journal={Ann. of Math. (2)},
      volume={63},
       pages={20\ndash 63},
         url={https://doi.org/10.2307/1969989},
}

\bib{Saksman_Soto_2017}{article}{
      author={Saksman, Eero},
      author={Soto, Tom\'{a}s},
       title={Traces of {B}esov, {T}riebel--{L}izorkin and {S}obolev spaces on
  metric spaces},
        date={2017},
     journal={Anal. Geom. Metr. Spaces},
      volume={5},
      number={1},
       pages={98\ndash 115},
         url={https://doi.org/10.1515/agms-2017-0006},
}

\bib{White-infima}{article}{
      author={White, Brian},
       title={Infima of energy functionals in homotopy classes of mappings},
        date={1986},
        ISSN={0022-040X},
     journal={J. Differential Geom.},
      volume={23},
      number={2},
       pages={127\ndash 142},
         url={http://projecteuclid.org/euclid.jdg/1214440023},
}

\bib{White-homotopy}{article}{
      author={White, Brian},
       title={Homotopy classes in {S}obolev spaces and the existence of energy
  minimizing maps},
        date={1988},
        ISSN={0001-5962},
     journal={Acta Math.},
      volume={160},
      number={1-2},
       pages={1\ndash 17},
         url={https://doi.org/10.1007/BF02392271},
}

\end{biblist}
\end{bibdiv}

\end{document}